\documentclass[10pt]{amsart}

\usepackage[T2A]{fontenc}
\usepackage{amsmath}
\usepackage{amsfonts}
\usepackage{amssymb}
\usepackage{amsthm}
\usepackage[toc,page]{appendix}
\usepackage{orcidlink}
\usepackage{amscd}

\usepackage[ukrainian,english]{babel}
\theoremstyle{plain}

\theoremstyle{definition}

\theoremstyle{plain}
\newtheorem*{theoremUkr*}{Theorem}
\newtheorem*{lemmaUkr*}{Lemma}
\newtheorem*{propositionUkr*}{Proposition}
\newtheorem*{statementUkr*}{Statement}
\newtheorem*{corollaryUkr*}{Corollary}
\theoremstyle{definition}
\newtheorem*{definitionUkr*}{Definition}
\theoremstyle{remark}
\newtheorem*{notationUkr*}{Notation}
\newtheorem*{remarkUkr*}{Remark}

\theoremstyle{plain}
\newtheorem{theoremEng}{Theorem}
\newtheorem{lemmaEng}[theoremEng]{Lemma}

\newtheorem{statementEng}[theoremEng]{Statement}
\newtheorem{corollaryEng}[theoremEng]{Corollary}
\theoremstyle{definition}

\newtheorem{remarkEng}[theoremEng]{Remark}
\numberwithin{theoremEng}{section}
\numberwithin{equation}{section}

\title[Uniqueness problem for accretive Schr\"{o}dinger operators]{Uniqueness problem for accretive Schr\"{o}dinger operators with complex singular coefficients}

\author[V. Mikhailets]{Vladimir Mikhailets}

\address{Vladimir Mikhailets
         \\ \orcidlinkf{0000-0002-1332-1562}
         }

\address{Department of Mathematics \\
         King's College London \\
         2CR2LS \\
         United Kingdom}

\address{Institute of Mathematics of NAS of Ukraine \\
         3, Tereschenkivska Str. \\
         Kyiv-4 \\
         Ukraine \\
         01024}

\email[Vladimir Mikhailets]{mikhailets@imath.kiev.ua, vladimir.mikhailets@gmail.com}

\author[V. Molyboga]{Volodymyr Molyboga}

\address{Volodymyr Molyboga
        \\ \orcidlinkf{0000-0002-5683-5694}
         }

\address{Institute of Mathematics of NAS of Ukraine \\
         3, Tereschenkivska Str. \\
         Kyiv-4 \\
         Ukraine \\
         01024}

\email[Volodymyr Molyboga]{molyboga@imath.kiev.ua, vm.imath@gmail.com}

\keywords{Hilbert space, Schr\"{o}dinger operator, accretive operator, singular coefficients, quasi-differential operator, uniqueness problem}

\subjclass[2020]{34B20, 34B24, 34L40}

\begin{document}

\begin{abstract}
This paper studies the uniqueness problem for the one-dimensional Schr\"{o}dinger operator associated with the formal differential expression
\begin{equation*}
l[u] =-u''+qu + i[(ru)'+ru'],
\end{equation*}
in the complex Hilbert space $L^{2}(\mathbb{R})$. The coefficients of the expression are complex-valued and satisfy
\begin{equation*}
q=s+Q', \quad  s \in L^1_{loc}\left(\mathbb{R}\right) \quad\text{and}\quad Q, r \in L^2_{loc}\left(\mathbb{R}\right),
\end{equation*}
where the derivative is understood in the sense of distributions. In particular, the potential $q$ can be a Radon measure on the line. With the help of specially selected quasi-derivatives, the expression $l$ is treated as a quasi-differential expression. The domains of the minimal $\mathrm{L}_{0}$ and maximal $\mathrm{L}$ operators associated with the expression $l$ in the space $L^{2}(\mathbb{R})$ are described. We find constructive conditions on the behaviour of $\mathrm{Im}\,r$ near $\pm \infty$ that guarantee that $\mathrm{L}_{0}=\mathrm{L}$ if the operator $\mathrm{L}_{0}$ is accretive.

We prove that these conditions are sharp even in the class of differential operators with smooth real-valued coefficients.
Examples are given to illustrate the main results of the paper.

\end{abstract}


\maketitle

\section{Introduction}\label{sec:Intro}
Let $\mathrm{H}$ be a complex Hilbert space with scalar product $(\cdot\,,\cdot)$. A linear operator $\mathrm{A}$ in the space $\mathrm{H}$ is called accretive \cite{Kato1995, Schm2012} if
\begin{equation*}
  \mathrm{Re}\,(\mathrm{A}x,x) \geqslant 0, \qquad \forall x\in \mathrm{Dom}(\mathrm{A}),
\end{equation*}
and maximal accretive if it has no nontrivial accretive extensions in the Hilbert space $\mathrm{H}$. An operator $\mathrm{A}$ is called $m$-accretive if the left complex half-plane belongs to the resolvent set $\rho(\mathrm{A})$ of $\mathrm{A}$. Such operators play an important role in the theory of semigroups. The condition that the strongly continuous operator semigroup $e^{-\mathrm{A}t}$, $t\geq 0$, is contractive is equivalent to the $m$-accretivity of the operator $\mathrm{A}$. If, in addition, the operator is also sectorial, then this semigroup is holomorphic.  The class of $m$-accretive operators is invariant with respect to the mapping $\mathrm{A} \mapsto \mathrm{A}^*$  \cite[Problem~V-3.31]{Kato1995}.

It follows from Phillips' theorem \cite[Theorem~1.1.1]{Phl1957} that a maximal accretive operator is either $m$-accretive or nondensely defined and nonclosable. Therefore, every densely defined accretive operator has a closure. A densely defined operator $\mathrm{A}$ in $\mathrm{H}$ is symmetric (self-adjoint) if and only if the operators $i\mathrm{A}$ and $-i\mathrm{A}$ are accretive ($m$-accretive). Thus, the problem of the $m$-accretivity of a linear operator encompasses the problem of the self-adjointness of a symmetric operator.

The problem of self-adjointness of symmetric Schr\"{o}dinger operators in Hilbert spaces $L^{2}(\mathbb{R}^{n})$, $n\in \mathbb{N}$, is inspired by problems of quantum mechanics. It has a long history and a huge bibliography (see monographs \cite{ReedSimon-book02_Eng_1975, AlbGszHKH2005, Zttl2005, Zttl2021, GszNclZnch2024} and works \cite{MkhSbl1999, Shbn2001, AlKsMl2010, HrMk2012, EcGsNcTs2013, KsMlNc2022} and the references therein). The obtained results are, in many cases, complete and form the basis of the spectral analysis of such operators and scattering theory. In mathematical models of real physical processes in strongly inhomogeneous media, differential operators with strongly singular coefficients naturally arise. As a rule, they can be interpreted as distributions from certain classes.

The study of properties of such operators is a complex mathematical problem due to difficulties of a fundamental character. They naturally arise even in the problem of the proper definition of such operators. Such a definition needs a description of functions on which the formal differential expression is given, and also the mapping defined by the expression.

In this paper, we introduce and study the class of Schr\"{o}dinger operators in the complex Hilbert space $L^{2}(\mathbb{R})$, which are associated with the formal differential expression
\begin{equation*}\label{fde_Shr_CMN}
  l[u]:=-u''+au'+bu,  \tag{$\ast$}
\end{equation*}
with complex-valued coefficients $a$ and $b$ satisfying the following conditions:
\begin{equation*}\label{eq_cond}
  b=s+B', \qquad s\in L_{loc}^{1}(\mathbb{R}) \quad  B, a \in L_{loc}^{2}(\mathbb{R}), 
\end{equation*}
with the derivative being understood in the sense of distributions.

Using the substitutions
\begin{equation*}\label{eq_CHNG}
  r:=\frac{1}{2i}a, \qquad Q:=B-\frac{1}{2}a,
\end{equation*}
the formal differential expression $(\ast)$ can be written in a more convenient form
\begin{equation}\label{fde_Shr}
l[u] =-u''+qu + i[(ru)'+ru'],
\end{equation}
where the complex-valued coefficients $q$ and $r$ satisfy the conditions
\begin{equation}\label{cond_Main}
q=s+Q', \quad  s \in L^1_{loc}\left(\mathbb{R}\right) \quad\text{and}\quad Q, r \in L^2_{loc}\left(\mathbb{R}\right).
\end{equation}
From now on, we assume without loss of generality that the formal differential expression takes the form~\eqref{fde_Shr}, and its coefficients satisfy conditions~\eqref{cond_Main}.

The article is structured as follows. In Section \ref{sec_ShrOp_mAccretivity} we establish our main uniqueness result (Theorem \ref{th_MAccretivityMAIN_A}), which contains sufficient conditions for the $m$-accretivity of minimal operators for $l$ and $l^+$, provided that the preminimal operators in the space $L^2(\mathbb{R})$ are accretive. To achieve this, in Section \ref{sec:SO} we develop a regularization technique using Shin--Zettl quasi-derivatives. We describe the maximal, pre-minimal and minimal operators associated with expressions $l$ and $l^+$ in $L^2(\mathbb{R})$.
Section~\ref{sec:IntTpTh} contains Theorem~\ref{th_MAccretivityMAIN_B}, which gives other conditions for the $m$-accretivity of minimal operators. These conditions are imposed only on the behaviour of $\mathrm{Im\,}r$ on some sequences of finite intervals that tend to $\pm \infty$ as $n\rightarrow\pm\infty$. 
Appendices A and B contain supplementary results on accretive differential operators with regular or smooth coefficients. Theorem \ref{th4} provides necessary and sufficient conditions for the accretivity of such operators. Theorem \ref{th5} describes a wide class of accretive Schr\"odinger operators with real-valued coefficients that fail to be $m$-accretive. Theorem \ref{th_ApendixB} shows that the sufficient conditions for $m$-accretivity in Theorem \ref{th_MAccretivityMAIN_A} are sharp, even in the case of smooth real-valued coefficients. Examples 1 and 2 illustrate the main results of the paper.

\section{1D Schr\"{o}dinger operators with singular coefficients}\label{sec:SO}
Let us consider the differential expression \eqref{fde_Shr} with complex-valued coefficients $q$ and $r$ on the line $\mathbb{R}$. If $q\in L_{loc}^{2}(\mathbb{R})$ and $r\in W_{2,loc}^{1}(\mathbb{R})$, then the differential expression $l$ is well defined on the set $C_{comp}^{\infty}(\mathbb{R})$ of infinitely smooth compactly supported functions on $\mathbb{R}$ (this set is dense in $L^{2}(\mathbb{R})$). Therefore, we associate the pre-minimal operator in the Hilbert space $L^{2}(\mathbb{R})$ with expression \eqref{fde_Shr}:
\begin{equation*}
  \mathrm{L}_{00}: u \rightarrow l[u], \qquad \mathrm{Dom}(\mathrm{L}_{00})=C_{comp}^{\infty}(\mathbb{R}),
\end{equation*}
and the maximal operator
\begin{equation*}
  \mathrm{L}: u \rightarrow l[u], \qquad \mathrm{Dom}(\mathrm{L})= \left\{u\in L^2(\mathbb{R})\cap W_{2,loc}^{2}(\mathbb{R})\,\left|\,l[u]\in   L^{2}(\mathbb{R})\right.\right\}.
\end{equation*}
It is clear that $\mathrm{L}_{00}\subset L$. The operator $\mathrm{L}$ is closed in the space $L^{2}(\mathbb{R})$ \cite{DnfSchw_II_1963}, therefore, the operator $\mathrm{L}_{00}$ is closable. Its closure $\tilde{\mathrm{L}}_{00}:=\mathrm{L}_{0}$ is called the minimal operator associated with expression \eqref{fde_Shr}. The formally adjoint differential expression to $l$ becomes
\begin{equation*}\label{fde_AShr}
l^{+}[v] =-v''+\overline{q}v + i[(\overline{r}v)'+\overline{r}v'],
\end{equation*}
where the overline denotes complex conjugation of the coefficients. Therefore, if the coefficients $q$ and $r$ in \eqref{fde_Shr} are real-valued, then $l=l^{+}$ and the operators $\mathrm{L}_{00}$ and $\mathrm{L}_{0}$ are symmetric. As is known, in this case (see \cite{NmrkRus1969, Schm2012})
\begin{equation*}
\mathrm{L}=\mathrm{L}_{0}^{\ast}=\mathrm{L}_{00}^{\ast}.
\end{equation*}
Similarly, we define the operators $\mathrm{L}_{00}^{+}$, $\mathrm{L}_{0}^{+}$ and $\mathrm{L}^{+}$ generated by the differential expression $l^{+}$ in the space $L^{2}(\mathbb{R})$.

The operators $\mathrm{L}_{00}$, $\mathrm{L}_{0}$, and $\mathrm{L}$ can also be defined in the case where $q\in L_{loc}^{1}(\mathbb{R})$ and $r\in \mathrm{AC}_{loc}(\mathbb{R})$. They act on functions as the differential expression $l$ \cite{NmrkRus1969, Wei1987, GrnMkhPnk2013}.

In this paper, we investigate the case when the coefficients of the expression $l$ are complex-valued and satisfy~\eqref{cond_Main}. In particular, if the function $Q$ has locally bounded variation, then the coefficient $q$ is a complex-valued Radon measure on a locally compact space $\mathbb{R}$. Schr\"{o}dinger operators with such potentials appear in many problems of modern mathematical physics (see, for example, \cite{AlbGszHKH2005, AlKsMl2010, EcGsNcTs2013, MkhMlb2018} and the bibliography therein). In this case, first of all, we should study the question about a reasonable definition of the operators generated by the formal differential expression $l$. The most natural, in our opinion, is the definition based on the regularization of the singular differential expression with the help of Shin--Zettl quasi-derivatives \cite{Shin1943, Zttl1975, EvMr1999}. The most extensively studied case is when  $r\equiv 0$ and $q=\overline{q}$, see \cite{AlKsMl2010, HrMk2012, MkhMlb2013, EcGsNcTs2013} and references therein. The more general case is considered in \cite{MkhGrnMlb2022}. In this case, the operator $\mathrm{L}_{0}=\mathrm{L}_{0}^{+}$ is symmetric and $\mathrm{L}=\mathrm{L}_{0}^{\ast}=(\mathrm{L}_{0}^{+})^{\ast}$.

In this paper, we investigate the general case of singular complex-valued coefficients based on the regularization of the formal differential expression $l$ using quasi-derivatives.


Define
\begin{equation*}
  G_{1}:=Q+i\,r \quad\text{and}\quad G_{2}:=Q-i\,r.
\end{equation*}
Then expression~\eqref{fde_Shr} can be presented in the more convenient form
\begin{equation}\label{fde_Shr_Work}
  l[u]=-u''+(G_{1}u)'-G_{2}u'+su
\end{equation}
since
\begin{equation*}
  (G_{1}u)'-G_{2}u'=Q'u+i\left[(ru)'+ru' \right].
\end{equation*}

Under conditions \eqref{cond_Main}, the Shin--Zettl matrix function is well defined $A(\cdot)\in L^1_{loc}$:
\begin{equation}\label{SZ_matrix}
A(\cdot)=
\begin{pmatrix}
G_{1} & 1 \\
-G_{1}G_{2} + s & -G_{2}
\end{pmatrix}
.
\end{equation}
The Shin-Zettl matrix for $l$ defines the quasi-derivatives as follows \cite{Shin1943, Zttl1975, EvMr1999}:
\begin{align}\label{quasiderivatives}
u^{[0]}:=u,\qquad
u^{[1]}:=u' - G_{1}u, \qquad
u^{[2]}:=\left(u^{[1]}\right)'+G_{2}u^{[1]}+\left(G_{1}G_{2} - s\right)u.
\end{align}
Using quasi-derivatives \eqref{quasiderivatives}, we properly define the formal Schr\"{o}dinger differential expression \eqref{fde_Shr} as a quasi-differential one by the formulas:
\begin{equation*}
l[u]:=-u^{[2]}\quad\text{with}\quad \mathrm{Dom}(l):=\left\{u:\mathbb{R}\rightarrow \mathbb{C}\left|\,u,\,u^{[1]}\in \mathrm{AC}_{loc}(\mathbb{R})\right.\right\}.
\end{equation*}

Now, in the complex Hilbert space $L^{2}(\mathbb{R})$, we define the Schr\"{o}dinger operator $\mathrm{L}$ generated by formal differential expression \eqref{fde_Shr} as a maximal quasi-differential operator
\begin{equation*}\label{maxOP}
  \mathrm{L}u:=l[u],\qquad
  \mathrm{Dom}(\mathrm{L}):= \left\{u\in L^2(\mathbb{R})\,\left|\,u,\,u^{[1]}\in \mathrm{AC}_{loc}(\mathbb{R}),l[u]\in   L^{2}(\mathbb{R})\right.\right\}.
\end{equation*}
The restriction of the maximal operator $\mathrm{L}$ to compactly supported functions generates a pre-minimal operator:
\begin{equation*}\label{preminOP}
  \mathrm{L}_{00}u:=\mathrm{L}u, \qquad
  \mathrm{Dom}(\mathrm{L}_{00}):=\left\{u\in \mathrm{Dom}(\mathrm{L})\,  \left|\,\mathrm{supp}\,u\Subset\mathbb{R}\right.\right\}.
\end{equation*}

Shin-Zettl matrix \eqref{SZ_matrix} defines the formally Lagrange adjoint quasi-differential expression $l^{+}$ as follows \cite{Shin1943, Zttl1975, EvMr1999}:
\begin{equation}\label{quasiderivatives+}
  v^{\{0\}}:=v, \qquad
  v^{\{1\}}:=v'-\overline{G_{2}}v, \qquad
  v^{\{2\}}:=\left(v^{\{1\}}\right)'+\overline{G_{1}}v^{\{1\}}+\left(\overline{G_{1}G_{2}}-\overline{s}\right)v,
\end{equation}
and
\begin{equation*}
    l^{+}[v]:=-v^{\{2\}}, \qquad
    \mathrm{Dom}(l^{+}):=\left\{v:\mathbb{R}\rightarrow \mathbb{C}\left|v,\,v^{\{1\}}\in\mathrm{AC}_{loc}(\mathbb{R})\right.\right\}.
\end{equation*}
The quasi-differential expression $l^{+}$ generates the following maximal and pre-minimal operators:
\begin{equation*}
    \mathrm{L}^{+}v:=l[v],\qquad
  \mathrm{Dom}(\mathrm{L}^{+}):=\left\{v\in L^{2}(\mathbb{R})\,\left|\,v,\,v^{\{1\}}\in \mathrm{AC}_{loc}(\mathbb{R}),l^{+}[v]\in L^{2}(\mathbb{R})\right.\right\},
\end{equation*}
and
\begin{equation*}
    \mathrm{L}_{00}^{+}v:=l[v], \qquad
    \mathrm{Dom}(\mathrm{L}_{00}^{+}):=\left\{v\in \mathrm{Dom}(\mathrm{L}^{+})\,\left|\,\mathrm{supp}\,v\Subset\mathbb{R}\right.\right\}.
\end{equation*}

Let us introduce the notation
\begin{align*}
  [u,v](x) & :=u(x)\overline{v^{\{1\}}(x)}-u^{[1]}(x)\overline{v(x)}, \\
  [u,v]_{a}^{b} & :=[u,v](b)-[u,v](a),\quad -\infty\leq a< b\leq \infty.
\end{align*}

The main result of this section is the following theorem.

\begin{theoremEng}\label{th_Properties}
The following statements hold:
\begin{itemize}
 \item [$1^{0}$.] The pre-minimal operators $\mathrm{L}_{00}$ and $\mathrm{L}_{00}^{+}$ are densely defined in the Hilbert space $L^{2}(\mathbb{R})$.
\item [$2^{0}$.] $\left(\mathrm{L}_{00}\right)^{\ast}=\mathrm{L}^{+}\quad\text{and}\quad \left(\mathrm{L}_{00}^{+}\right)^{\ast}=\mathrm{L}.$
\item [$3^{0}$.] The operators $\mathrm{L}$, $\mathrm{L}^{+}$ are closed, and the pre-minimal operators $\mathrm{L}_{00}$, $\mathrm{L}_{00}^{+}$ admit a closure.
\item [$4^{0}$.] Domains of the minimal operators $\mathrm{L}_{0}$, $\mathrm{L}_{0}^{+}$ allow the following descriptions:
\begin{equation*}
  \mathrm{Dom}(\mathrm{L}_{0})  =\left\{u\in \mathrm{Dom}(\mathrm{L}) \left|\,[u,v]_{-\infty}^{\infty}=0\quad \forall
v\in  \mathrm{Dom}(\mathrm{L}^{+})\right.\right\}
\end{equation*}
and
\begin{equation*}
  \mathrm{Dom}(\mathrm{L}_{0}^{+})  =\left\{v\in \mathrm{Dom}(\mathrm{L}^{+}) \left|\,[u,v]_{-\infty}^{\infty}=0\quad
\forall u\in  \mathrm{Dom}(\mathrm{L})\right.\right\},
\end{equation*}
with
\begin{align*}
  i)\hspace{5pt} & \mathrm{Dom}(\mathrm{L})\subset W_{2,loc}^{1}(\mathbb{R}) \quad \text{and} \quad \mathrm{Dom}(\mathrm{L^{+}})\subset W_{2,loc}^{1}(\mathbb{R}), \\
  ii)\hspace{5pt} & \mathrm{Dom}(\mathrm{L_{00}})\subset W_{2,comp}^{1}(\mathbb{R}) \quad \text{and} \quad \mathrm{Dom}(\mathrm{L_{00}^{+}})\subset W_{2,comp}^{1}(\mathbb{R}). \hspace{90pt}
\end{align*}
 \item [$5^{0}$.] The quadratic form associated with the pre-minimal operator is expressed as follows
\begin{equation*}
 (\mathrm{L}_{00}u,u)_{L^2(\mathbb{R})}=\int_{\mathbb{R}}|u'|^{2}d\,x- \int_{\mathbb{R}}\left(G_{1}u\overline{u}'+G_{2}u'\overline{u}\right)d\,x+\int_{\mathbb{R}}s|u|^{2}d\,x, \qquad
 u\in \mathrm{Dom}(\mathrm{L}_{00}).
\end{equation*}
\end{itemize}
\end{theoremEng}


The proof of Theorem~\ref{th_Properties} relies on the following statements.
\begin{lemmaEng}[Generalized Lagrange identity {\cite[Corollary~1]{Zttl1975}}]\label{lm_GnrlLgrnIdnt_FI}
For arbitrary functions $u\in \mathrm{Dom}(\mathrm{L})$ and $v\in \mathrm{Dom}(\mathrm{L}^{+})$, the following relation holds:
\begin{equation}\label{eq_GnrlLgnrIdnt_FI}
\int_{a}^{b}l[u]\overline{v}d\,x-\int_{a}^{b}u\overline{l^{+}[v]}d\,x=[u,v]_{a}^{b}
\end{equation}
whenever $-\infty<a<b<+\infty$.
\end{lemmaEng}

\begin{lemmaEng}\label{lm_uv_form_lim}
For arbitrary functions $u\in \mathrm{Dom}(\mathrm{L})$ and $v\in \mathrm{Dom}(\mathrm{L}^{+})$ the following finite limits exist:
\begin{equation*}\label{eq2_uv_form_lim}
[u,v](-\infty):=\lim_{x\rightarrow-\infty}[u,v](x),\qquad [u,v](\infty):=\lim_{x\rightarrow\infty}[u,v](x).
\end{equation*}
\end{lemmaEng}
\begin{proof}
Let us fix $b$ in relation \eqref{eq_GnrlLgnrIdnt_FI}, and proceed to the limit at $a\rightarrow-\infty$. Since $u,v,l[u],l^{+}[v]\in L^{2}(\mathbb{R})$ by the assumptions of the lemma, the limit $[u,v](-\infty)$ exists and is finite. The existence and finiteness of the limit $[u,v](\infty)$ are proved  similarly.
\end{proof}

Combining Lemmas~\ref{lm_GnrlLgrnIdnt_FI} and~\ref{lm_uv_form_lim} yields the following integral identity.
\begin{statementEng}\label{st_GnrlLgnrIdnt}
For arbitrary $u\in \mathrm{Dom}(\mathrm{L})$ and $v\in \mathrm{Dom}(\mathrm{L}^{+})$, the following identity holds:
\begin{equation}\label{eq_GLI}
 \int_{-\infty}^{\infty}l[u]\overline{v}d\,x-\int_{-\infty}^{\infty}u\overline{l^{+}[v]}d\,x=[u,v]_{-\infty}^{\infty},
\end{equation}
where the limits
\begin{equation*}\label{eq1_uv_form_lim}
[u,v](-\infty):=\lim_{x\rightarrow-\infty}[u,v](x),\qquad [u,v](\infty):=\lim_{t\rightarrow\infty}[u,v](x),\end{equation*}
exist and are finite.
\end{statementEng}

Let $\Delta=(a,b)$ with $-\infty<a<b<\infty$ be a finite interval on $\mathbb{R}$.

In the complex Hilbert space $L^{2}(\Delta)$, let us consider the operators generated by the formal Schr\"{o}dinger differential expression \eqref{fde_Shr}. Basic assumptions \eqref{cond_Main} regarding the coefficients on a finite interval take the form
\begin{equation*}\label{cond_Main_FiniteInterval}
q=s+Q', \quad  Q, r \in L^2\left(\Delta\right) \quad\text{and}\quad s \in L^1\left(\Delta\right).
\end{equation*}
Therefore, the corresponding Shin-Zettl matrix-valued function \eqref{SZ_matrix} is well posed, and the corresponding quasi-derivatives \eqref{quasiderivatives} and \eqref{quasiderivatives+} are well defined.

Let us consider the following quasi-differential expressions $l$ and $l^{+}$ on a finite interval $\Delta$:
\begin{equation*}
  l_{\Delta}[u]  :=-u^{[2]},\qquad  \mathrm{Dom}(l_\Delta)  :=\left\{u:\overline{\Delta}\rightarrow \mathbb{C}\left|\,u,\,u^{[1]}\in \mathrm{AC}(\overline{\Delta})\right.\right\},
\end{equation*}
and
\begin{equation*}
  l_{\Delta}^{+}[v]  :=-v^{\{2\}},\qquad  \mathrm{Dom}(l_\Delta^{+})  :=\left\{v:\overline{\Delta}\rightarrow \mathbb{C}\left|\,v,\,v^{\{1\}}\in \mathrm{AC}(\overline{\Delta})\right.\right\}.
\end{equation*}
They naturally generate the minimal and maximal operators in the space $L^2(\Delta)$:
\begin{align*}
\mathrm{L}_\Delta u & :=l_\Delta[u], \qquad
& \mathrm{Dom}(\mathrm{L}_\Delta) & :=\left\{u\in L^2(\Delta)\left|\,u,\,u^{[1]}\in\mathrm{AC}(\overline{\Delta}),l_\Delta[u]\in L^2(\Delta)\right.\right\}, \\
\mathrm{L}_{0\Delta}u & :=\mathrm{L}_\Delta u, \qquad
& \mathrm{Dom}(\mathrm{L}_{0\Delta}) & :=\left\{u\in \mathrm{Dom}(\mathrm{L}_\Delta) \left|\,u^{[j]}(a)=u^{[j]}(b)=0,\; j=0,1\right.\right\},
\end{align*}
and
\begin{align*}
  \mathrm{L}_{\Delta}^{+}v & :=l_{\Delta}^{+}[v], \qquad
  & \mathrm{Dom}(\mathrm{L}_{\Delta}^{+}) & :=\left\{v\in L^{2}(\Delta) \left|\,v,\,v^{\{1\}}\in \mathrm{AC}(\overline{\Delta}), l_{\Delta}^{+}[v]\in L^{2}(\Delta)\right.\right\}, \\
  \mathrm{L}_{0\Delta}^{+}v & :=l_{\Delta}^{+}[v], \qquad
  & \mathrm{Dom}(\mathrm{L}_{0\Delta}^{+}) & :=\left\{v\in \mathrm{Dom}(\mathrm{L}_{\Delta}^{+}) \left|\,v^{\{1\}}(a)=v^{\{1\}}(b)=0,\; j=0,1\right.\right\}.
\end{align*}

Let us recall a well-known result.
\begin{statementEng}[{\cite[Theorem~10,~11]{Zttl1975}}]\label{st_QDO_PropertiesFI}
The minimal operators $\mathrm{L}_{0\Delta}$ and $\mathrm{L}_{0\Delta}^{+}$ are densely defined in the Hilbert space $L^{2}(\Delta)$, and the following statements hold:
\begin{align*}
   & (a)\hspace{5pt} (\mathrm{L}_{0\Delta})^{\ast}=\mathrm{L}_{\Delta}^{+}, \hspace{50pt} && (c)\hspace{5pt} \mathrm{L}_{0\Delta}=(\mathrm{L}_{\Delta}^{+})^{\ast},\hspace{175pt} \\
   & (b)\hspace{5pt} (\mathrm{L}_{0\Delta}^{+})^{\ast}=\mathrm{L}_{\Delta}, \hspace{50pt} && (d)\hspace{5pt} \mathrm{L}_{0\Delta}^{+}=(\mathrm{L}_{\Delta})^{\ast}.
\end{align*}
\end{statementEng}

\begin{proof}[Proof of Theorem~\ref{th_Properties}]
The proof of the properties $1^{0}$ -- $4 ^{0}$ is similar to the case of the semi-axis and formally self-adjoint differential expressions $l=l^{+}$ \cite{Zttl1975, NmrkRus1969}. In their proof, the properties of operators defined on a finite interval are used. For the reader's convenience, we will give the proof.


$1^{0}$. Let us prove the density of the domain of the pre-minimal operator $\mathrm{Dom}(\mathrm{L}_{00})$ in the Hilbert space $L^2(\mathbb{R})$.

Suppose that a function $h\in L^2(\mathbb{R})$ is orthogonal to $\mathrm{Dom}(\mathrm{L}_{00})$, and show that $h\equiv0$, which will mean the density of $\mathrm{Dom}(\mathrm{L}_{00})$ in the space $L^2(\mathbb{R})$.

Let $\Delta$ be an arbitrary fixed finite interval. The quasi-differential expression $l$ defines operators $\mathrm{L}_{0\Delta}$ and $\mathrm{L}_\Delta$. The space $L^{2}(\Delta)$ can be embedded in the space $L^{2}(\mathbb{R})$ if we assume that outside the interval $\overline{\Delta}$, the function $u\in L^{2}(\Delta)$ is zero. In this way, the domain $\mathrm{Dom}(\mathrm{L}_{0\Delta})$ of the operator $\mathrm{L}_{0\Delta}$ takes part of $\mathrm{Dom}(\mathrm{L})$. Moreover, the function thus extended belongs to $\mathrm{Dom}(\mathrm{L}_{00})$.

Thus, the function $h$ is orthogonal to $\mathrm{Dom}(\mathrm{L}_{0\Delta})$. Then, by  Statement~\ref{st_QDO_PropertiesFI}, the set $\mathrm{Dom}(\mathrm{L}_{0\Delta})$ is dense in the space $L^2(\Delta)$. Therefore, $h|_{\overline{\Delta}}=0$ almost everywhere.

Since the interval $\Delta\subset \mathbb{R}$ is chosen arbitrarily, we have $h=0$ almost everywhere on $\mathbb{R}$, which is equivalent to the density of $\mathrm{Dom}(\mathrm{L}_{00})$ in $L^2(\mathbb{R})$.

The density of the domain of the pre-minimal operator $\mathrm{Dom}(\mathrm{L}_{00}^{+})$ in $L^2(\mathbb{R})$ is proved similarly.


$2^{0}$. Let us prove the first relation
\begin{equation*}
\left(\mathrm{L}_{00}\right)^{\ast}=\mathrm{L}^{+}.
\end{equation*}
The second relation is proved similarly.

It follows from $1^{0}$ that for the operator $\mathrm{L}_{00}$, there exists the adjoint operator $(\mathrm{L}_{00})^{\ast}$.

Let $u\in \mathrm{Dom}(\mathrm{L}_{00})$ and $v\in \mathrm{Dom}(\mathrm{L}^{+})$. Then, using the generalized Lagrange identity~\eqref{eq_GLI}, we obtain
\begin{equation*}\label{eq_Prp20}
(\mathrm{L}_{00}u,v)_{L^2(\mathbb{R})}=(u,\mathrm{L}^{+}v)_{L^2(\mathbb{R})}.
\end{equation*}
Therefore $\mathrm{L}^{+}\subset (\mathrm{L}_{00})^*$. Let us prove the inverse inclusion.

Let $v$ be an arbitrary element of $\mathrm{Dom}\left((\mathrm{L}_{00})^*\right)$, and let $\Delta=(\alpha,\beta)$ be a fixed finite interval on $\mathbb{R}$. Then
\begin{equation*}\label{eq_Prp22}
\left((\mathrm{L}_{00})^*v,u\right)_{L^2(\mathbb{R})}=\left(v,\mathrm{L}_{00}u\right)_{L^2(\mathbb{R})}\quad\text{for every}\quad u\in \mathrm{Dom}(\mathrm{L}_{0\Delta}).
\end{equation*}
Since $u|_{\mathbb{R}\setminus\overline{\Delta}}=0$, the scalar products are represented as integrals over the interval $\overline{\Delta}$, that is, they are scalar products in $L^{2}(\Delta)$, and
\begin{equation}\label{eq_Prp24}
\left(((\mathrm{L}_{00})^*v)_\Delta,u\right)_{L^2(\Delta)}=\left(v_\Delta,\mathrm{L}_{0\Delta}u\right)_{L^2(\Delta)} \quad\text{for every}\quad u\in \mathrm{Dom}(\mathrm{L}_{0\Delta}).
\end{equation}
Here $((\mathrm{L}_{00})^*v)_\Delta$ and $v_\Delta$ are restrictions of the functions to the interval~$\overline{\Delta}$.

Using \eqref{eq_Prp24} and applying the generalized Lagrange identity~\eqref{eq_GnrlLgnrIdnt_FI}, we obtain
\begin{equation*}
((\mathrm{L}_{00})^*v)_\Delta=\mathrm{L}_{\Delta}^{+} v_\Delta=\left(l^{+}[v]\right)_\Delta.
\end{equation*}
Since the interval $\Delta$ can be chosen arbitrarily, then
\begin{equation*}
v\in \mathrm{Dom}(\mathrm{L}^{+}) \quad\text{and}\quad (\mathrm{L}_{00})^{\ast}v=l^{+}[v]=\mathrm{L}^{+}v,
\end{equation*}
which completes the proof.


$3^{0}$. This statement is a direct consequence of property~$2^{0}$.

$4^{0}$. Since the operator $\mathrm{L}_0$ is closed, and because
\begin{equation*}
  (\mathrm{L}_{0})^{\ast}=\mathrm{L}^{+}\quad\text{and}\quad \mathrm{L}_{0}=(\mathrm{L}^{+})^{\ast},
\end{equation*}
it is obvious that the set $\mathrm{Dom}\left(\mathrm{L}_0\right)$ consists of those and only those functions $u\in \mathrm{Dom}\left(\mathrm{L}\right)$ that satisfy the relation
\begin{equation*}
(u,\mathrm{L}^{+}v)_{L^{2}(\mathbb{R})}=(\mathrm{L}u,v)_{L^{2}(\mathbb{R})}\quad\text{for every}\quad  v\in \mathrm{Dom}\left(\mathrm{L}^{+}\right),
\end{equation*}
which, taking into account the generalized Lagrange identity on $\mathbb{R}$, is equivalent to statement~$4^{0}$.


$5^{0}$. By assumptions~\eqref{cond_Main}, we have $G_{1}, G_{2} \in L_{loc}^2(\mathbb{R})$ and $\mathrm{Dom}(\mathrm{L}_{00})\subset W^1_{2,comp}(\mathbb{R})$. Hence, integrating by parts, we obtain for any $u\in\mathrm{Dom}(\mathrm{L}_{00})$ the following:
\begin{align*}
  (\mathrm{L}_{00}u,u)_{L^{2}(\mathbb{R})} & =\int_{\mathbb{R}}\left(-(u^{[1]})'\overline{u}-G_{2}u^{[1]}\overline{u}-G_{1}G_{2}|u|^{2}+s|u|^{2}\right)d\,x   \\
  & =\int_{\mathbb{R}}\left((u^{[1]}\overline{u}'-G_{2}u^{[1]}\overline{u}-G_{1}G_{2}|u|^{2}+s|u|^{2}\right)d\,x \\
  & =\int_{\mathbb{R}}\left((u'-G_{1}u)\overline{u}'-G_{2}(u'-G_{1}u)\overline{u}-G_{1}G_{2}|u|^{2}+s|u|^{2}\right)d\,x \\
  & =\int_{\mathbb{R}}\left(|u'|^{2}-\left(G_{1}u\overline{u}'+G_{2}u'\overline{u}\right)+s|u|^{2}\right)d\,x.
\end{align*}


The proof of Theorem~\ref{th_Properties} is complete.

\end{proof}


\section{Uniqueness Problem}\label{sec_ShrOp_mAccretivity}
The uniqueness problem for a differential expression $l$ in the Hilbert space $L^{2}(\mathbb{R})$ consists in finding conditions on its coefficients $q$ and $r$ under which the equality $\mathrm{L}_{0}=\mathrm{L}$ holds. If these functions are real-valued, then the operator $\mathrm{L}_{0}$ is symmetric, and $\mathrm{L}= \mathrm{L}_{0}^{\ast}$. We thus deal with the self-adjointness conditions for the symmetric operator $\mathrm{L}_{0}$ in the Hilbert space $L^{2}(\mathbb{R})$. In the case when $r\equiv 0$ and $q\in L_{loc}^{1}(\mathbb{R})$, this topic was studied in detail by many analysts. In particular, H.~Weyl~\cite{Weyl1910} was the first to prove the self-adjointness of the operator $\mathrm{L}_{0}$ in the case where $r\equiv 0$ and when the function $q$ is continuous and bounded from below on $\mathbb{R}$. It is easy to see that the operator $\mathrm{L}_{0}$ is bounded below in this case. In this connection, the hypothesis arose that this abstract condition already guarantees the self-adjointness of the operator $\mathrm{L}_{0}$. It turned out to be true even for the more general Sturm--Liouville differential expression
\begin{equation*}
l[u] =-(pu')'+qu + i[(ru)'+ru'].
\end{equation*}
Hartman~\cite{Hrtmn1948} and Rellich~\cite{Rllch1951} showed that $\mathrm{L}_{0}= \mathrm{L}$ if
\begin{equation*}\label{cond_HR}
\int_{0}^{\infty}p^{-1/2}(t)d\,t=\int_{-\infty}^{0}p^{-1/2}(t)d\,t=\infty,
\end{equation*}
$r\equiv 0$, $0<p\in C^2(\mathbb{R})$ and the function $q$ is piecewise continuous on $\mathbb{R}$. Under more general assumptions
\begin{equation*}
  q=Q'+s,\quad \frac{1}{\sqrt{|p|}}, \frac{Q}{\sqrt{|p|}}, \frac{r}{\sqrt{|p|}} \in L^2_{loc}\left(\mathbb{R}\right), \quad s \in L^1_{loc}\left(\mathbb{R}\right), \quad
  0<p\in W_{2,loc}^{1}(\mathbb{R}),
\end{equation*}
this fact was proven by the authors in \cite{MkhGrnMlb2022}.

Similar theorems have been proved for second-order elliptic differential operators in the spaces $L^{2}(\mathbb{R}^{n})$, with $n\in \mathbb{N}$, and other classes of operators (see, for example, \cite{Orchk1988, Pvznr1953, Shbn2001, StHnsn1966, Wnhltz1958} and the references therein).

The case where the symmetric operator $\mathrm{L}_{0}$ is not bounded from below was also studied quite thoroughly. In particular, Sears~\cite{Sears1950} showed that if $r\equiv 0$, $q\in C(\mathbb{R})$, and if there exists an even continuous function $Q(\cdot)\geq \delta>0$ on $\mathbb{R}$ such that
\begin{equation*}
  q(x)\geq -Q(x), \qquad \int_{0}^{\infty}Q^{-1/2}(x)dx=\infty,
\end{equation*}
furthermore, either the function $Q$ is monotone on $[0;\infty)$ or
\begin{equation*}
  \limsup_{x\rightarrow\infty}\frac{Q'(x)}{\{Q(x)\}^{3/2}}< \infty,
\end{equation*}
then the operator $\mathrm{L}_{0}$ is self-adjoint. See also papers \cite{Hrtmn1948, Ismgl1963, Orchk1974} and references therein.

The case when the coefficients of the differential expression $l$ are complex-valued is more complicated because the operator $\mathrm{L}_{0}$ in general is non-symmetric. The question about the conditions for its $m$-accretivity, to the best of the authors' knowledge, has not been previously studied, even in the case of continuous coefficients $q$ and $r$.

In this paper, we investigate the case when these coefficients are singular and satisfy assumption \eqref{cond_Main}. To this end, we develop a methodology of \cite{MkhGrnMlb2022}, where symmetric Sturm--Liouville operators in $L^{2}(\mathbb{R})$ were studied. The main result of this section is Theorem~\ref{th_MAccretivityMAIN_A}. Its proof is based on the results of Section~\ref{sec:SO}.

Let us introduce some notation. Given a real-valued function $f\in L_{loc}^{1}(\mathrm{R})$, we let
\begin{align*}
    f^{+} & :=\frac{|f|+f}{2}\geq 0, \qquad  f^{-}:=\frac{|f|-f}{2}\geq 0, \qquad f=f^{+}-f^{-}, \\
    r & =r_{0}+ir_{1}, \qquad \text{where}\qquad r_{0}:=\mathrm{Re}\,r, \quad r_{1}:=\mathrm{Im}\,r.
\end{align*}

\begin{theoremEng}\label{th_MAccretivityMAIN_A}
Let the coefficients of the formal differential Schr\"{o}dinger expression \eqref{fde_Shr} satisfy conditions \eqref{cond_Main}, and suppose that the pre-minimal operators $\mathrm{L}_{00}$ and $\mathrm{L}_{00}^{+}$ are accretive and
\begin{equation*}
  r_{1}^{\pm}(x)=O(m(x)) \quad\text{as}\quad x\rightarrow \mp\infty,
\end{equation*}
where the function $m(x)$ satisfies
\begin{align*}
  (\mathrm{m}_{1})\hspace{5pt} & 1\leq m(x)\in W_{2,loc}^{1}(\mathbb{R}), \\
  (\mathrm{m}_{2})\hspace{5pt} & \int\limits_{-\infty}^{0}\frac{d\,s}{m(s)}=\int\limits_{0}^{+\infty}\frac{d\,s}{m(s)}=+\infty. \hspace{315pt}
\end{align*}
Then the minimal operators $\mathrm{L}_{0}$ and $\mathrm{L}_{0}^{+}$ are $m$-accretive, i. e., $\mathrm{L}_{0}=\mathrm{L}$ and $\mathrm{L}_{0}^{+}=\mathrm{L}^{+}$.
\end{theoremEng}

\begin{remarkEng}
In the case where $q\in L_{loc}^{2}(\mathbb{R})$ and $r\in W_{2,loc}^{1}(\mathbb{R})$, the operators $\mathrm{L}_{00}$ and $\mathrm{L}_{00}^{+}$ are accretive if and only if both are accretive. Therefore, the accretivity assumption for $\mathrm{L}_{00}^{+}$ in Theorem~\ref{th_MAccretivityMAIN_A} can be removed in this case.
\end{remarkEng}

\begin{remarkEng}
From the equalities $\left(\mathrm{L}_{00}\right)^{\ast}=\mathrm{L}^{+}$ and $\left(\mathrm{L}_{00}^{+}\right)^{\ast}=\mathrm{L}$ of Theorem \ref{th_Properties} it follows that the accretivity of the operator $\mathrm{L}_{00}^{+}$ is a necessary condition for the $m$-accretivity of the operator $L_0$.
\end{remarkEng}

\begin{corollaryEng}\label{cor_ThA_10}
Let $\Lambda$ be an arbitrary closed sector of the half-plane $\mathrm{Re}\,\lambda\geq 0$ with vertex at the point $0$, and let assumptions $(\mathrm{m}_{1})$, $(\mathrm{m}_{2})$ of Theorem~\ref{th_MAccretivityMAIN_A} be satisfied. If
\begin{equation*}
  (\mathrm{L}_{00}u,u)_{L^2(\mathbb{R})}\in \Lambda\quad\text{for every}\quad u \in \mathrm{Dom}(\mathrm{L}_{00}), \quad\text{and if}\quad (\mathrm{L}_{00}^{+}v,v)_{L^2(\mathbb{R})}\in \overline{\Lambda}\quad\text{for every}\quad v \in \mathrm{Dom}(\mathrm{L}_{00}^{+}),
\end{equation*}
then $\rho(\mathrm{L}_{0})\supset\mathbb{C}\setminus \Lambda$ and $\rho(\mathrm{L}_{0}^{+})\supset\mathbb{C}\setminus \overline{\Lambda}$.
\end{corollaryEng}

First, we will prove some necessary technical results.
\begin{lemmaEng}\label{lm_phi_u}
Let initial assumptions~\eqref{cond_Main} be satisfied. Then, for an arbitrary compactly supported function $\varphi\in W_{2,comp}^{2}(\mathbb{R})$ and for all $u\in \mathrm{Dom}(\mathrm{L})$ and $v\in \mathrm{Dom}(\mathrm{L}^{+})$, the  following statements hold:
\begin{align*}
 \mathtt{(i)}\hspace{5pt} & \varphi u  \in \mathrm{AC}_{comp}(\mathbb{R}), \\
& \varphi v  \in \mathrm{AC}_{comp}(\mathbb{R}); \\
\mathtt{(ii)}\hspace{5pt} & (\varphi u)^{[1]}  =\varphi' u+ \varphi u^{[1]}\in \mathrm{AC}_{comp}(\mathbb{R}), \\
& (\varphi v)^{[1]}  =\varphi' v+ \varphi v^{\{1\}}\in \mathrm{AC}_{comp}(\mathbb{R}); \\
\mathtt{(iii)}\hspace{5pt} &  l[\varphi u] =\varphi l[u]-\varphi'' u-2\varphi' u'+\left(G_{1}-G_{2}\right)\varphi' u \in L_{comp}^{2}(\mathbb{R}), \\
 & l^{+}[\varphi v]  =\varphi l^{+}[v]-\varphi'' v-2\varphi' v'+\left(\overline{G_{2}}-\overline{G_{1}}\right)\varphi' v \in L_{comp}^{2}(\mathbb{R}); \\
\mathtt{(iv)}\hspace{5pt} & \mathrm{supp}\,\varphi u  \text{ is a compact set}, \\
 & \mathrm{supp}\,\varphi v  \text{ is a compact set}, \hspace{275pt}
\end{align*}
and
\begin{equation*}
  \varphi u\in \mathrm{Dom}(\mathrm{L}_{00})\subset W_{2,comp}^{1}(\mathbb{R}) \quad \text{and} \quad \varphi v\in \mathrm{Dom}(\mathrm{L}_{00}^{+})\subset W_{2,comp}^{1}(\mathbb{R}).
\end{equation*}

In particular, if $a'<a<b<b'$ and if the function $\varphi\in W_{2,comp}^{2}(\mathbb{R})$ satisfies the conditions
\begin{align*}
\mathtt{(a)}\hspace{5pt} & 0 \leq\varphi(x) \leq 1,\quad x\in \mathbb{R},\hspace{295pt} \\
\mathtt{(b)}\hspace{5pt} & \varphi(x) = \begin{cases}
1, & x\in [a;b];\\
0, &x\in (\infty; a')\bigcup (b';\infty),
\end{cases}
\end{align*}
then the conclusions of the lemma hold.
\end{lemmaEng}
\begin{proof}
It can be proved by direct calculations.
\end{proof}

\begin{proof}[Proof of Theorem~\ref{th_MAccretivityMAIN_A}]
Let the assumptions of the theorem be satisfied. Without loss of generality, we assume that
\begin{equation}\label{eq_MA20}
\mathrm{Re}\,(\mathrm{L}_{00}u,u)_{L^2(\mathbb{R})}\geq (u,u)_{L^2(\mathbb{R})}\quad\text{for every}\quad u\in \mathrm{Dom}(\mathrm{L}_{00}),
\end{equation}
and that
\begin{equation}\label{eq_MA22}
\mathrm{Re}\,(v,\mathrm{L}_{00}^{+} v)_{L^2(\mathbb{R})}\geq (v,v)_{L^2(\mathbb{R})}\quad\text{for every}\quad v\in \mathrm{Dom}(\mathrm{L}_{00}^{+}).
\end{equation}

To prove that the minimal operator $\mathrm{L}_0$ is $m$-accretive, it suffices to show that the null space $\mathcal{N}\,(\mathrm{L}^{+})$ of the maximal operator $\mathrm{L}^{+}$ contains only the zero element, namely
\begin{equation*}
 \mathrm{dim}\,(\mathcal{N}\,(\mathrm{L}^{+}))=\mathrm{dim}\,(\mathrm{rank}\,\mathrm{L}_{0})^{\perp}=\mathrm{def}\,\mathrm{L}_{0}=0.
\end{equation*}
Then, an arbitrary point of the half-plane $\mathrm{Rm}\,\lambda<1$ is also regular for the operator $\mathrm{L}_0$ because the defect numbers of the closed operator are constant in each connected component of a set of points of regular type of this operator (see, for example, \cite{Glzmn1966} and \cite[Theorem~3.2]{Kato1995}).

Let $v\in L^{2}(\mathbb{R})$ be a solution of the equation
\begin{equation}\label{eq_MA24}
  \mathrm{L}^{+}v=0.
\end{equation}
Let us show that $v\equiv 0$.

\begin{lemmaEng}\label{lm_MainWork_Inq}
Let the general assumptions~\eqref{cond_Main} be satisfied, and let $v\in L^{2}(\mathbb{R})$ be a solution of the equation
\begin{equation*}
  \mathrm{L}^{+}v=0.
\end{equation*}
Then, for an arbitrary compactly supported function  $\varphi\in W_{2,comp}^{2}(\mathbb{R})$ with $\mathrm{Im}\,\varphi=0$ the following inequality holds:
\begin{equation}\label{eq_MA_MainWorkInq}
  \mathrm{Re}\,\left(\varphi v, \mathrm{L}_{00}^{+}\varphi v\right)_{L^{2}(\mathbb{R})}= \int_{\mathbb{R}}(\varphi')^{2}|v|^{2}d\,x+2\int_{\mathbb{R}}r_{1}\varphi'\varphi|v|^{2}d\,x \geq \int_{\mathbb{R}} \varphi^{2}|v|^{2}d\,x.
\end{equation}
\end{lemmaEng}
\begin{proof}
Given a function $\varphi\in W_{2,comp}^{2}(\mathbb{R})$ with $\mathrm{Im}\,\varphi=0$ and taking into account that
\begin{equation*}
  \varphi v\in \mathrm{Dom}(\mathrm{L}_{00}^{+})\subset W_{2,comp}^{1}(\mathbb{R}),
\end{equation*}
and that $\mathrm{l}^{+}[v]=0$, we obtain by direct calculations (see Lemma~\ref{lm_phi_u}) the following:
\begin{align*}
  \left(\varphi v, \mathrm{L}_{00}^{+}\varphi v\right)_{L^{2}(\mathbb{R})} & =\left(\varphi v,\mathrm{l}^{+}[\varphi v]\right)_{L^{2}(\mathbb{R})}=\int_{\mathbb{R}}\left(-\varphi'' \varphi|v|^{2}-2\varphi'\varphi v \overline{v'}+\left(G_{2}-G_{1}\right)\varphi'\varphi|v|^{2} \right)d\,x.
\end{align*}
Then, applying the integration by parts formula to the first term (recall that the function $\varphi$ has compact support), we obtain
\begin{align*}
  \left(\varphi v, \mathrm{L}_{00}^{+}\varphi v\right)_{L^{2}(\mathbb{R})} & = \int_{\mathbb{R}}\left(\varphi' \varphi'|v|^{2}+\varphi'\varphi v'\overline{v}+\varphi'\varphi v\overline{v}'-2\varphi'\varphi v\overline{v'}+\left(G_{2}-G_{1}\right)\varphi'\varphi|v|^{2} \right)d\,x \\
   & =\int_{\mathbb{R}}\left((\varphi')^{2}|v|^{2}+ \varphi'\varphi(v'\overline{v}-v\overline{v}')+\left(G_{2}-G_{1}\right)\varphi'\varphi|v|^{2} \right)d\,x.
\end{align*}
Therefore, for any $\varphi\in W_{2,comp}^{2}(\mathbb{R})$ with $\mathrm{Im}\,\varphi=0$ and every solution $v$ of the equation $\mathrm{L}^{+}v=0$, we have
\begin{equation*}\label{eq_MA26}
  \left(\varphi v, \mathrm{L}_{00}^{+}\varphi v\right)_{L^{2}(\mathbb{R})} = \int_{\mathbb{R}}(\varphi')^{2}|v|^{2}d\,x+ \int_{\mathbb{R}}\varphi'\varphi(v'\overline{v}-v\overline{v}')d\,x-2i\int_{\mathbb{R}}r\varphi'\varphi|v|^{2}d\,x,
\end{equation*}
and (recall that $r=r_{0}+ir_{1}$)
\begin{align}\label{eq_MA28}
  \left(\varphi v, \mathrm{L}_{00}^{+}\varphi v\right)_{L^{2}(\mathbb{R})} & = \int_{\mathbb{R}}(\varphi')^{2}|v|^{2}d\,x+2\int_{\mathbb{R}} r_{1}\varphi'\varphi|v|^{2}d\,x \notag \\
  & +\int_{\mathbb{R}}\varphi'\varphi(v'\overline{v}-v\overline{v}')d\,x-2i\int_{\mathbb{R}} r_{0}\varphi'\varphi|v|^{2}d\,x.
\end{align}
Since
\begin{align*}
    \mathrm{Re}\,\int_{\mathbb{R}}\varphi'\varphi(v'\overline{v}-v\overline{v}')d\,x=0 \quad\text{and}\quad
    \mathrm{Re}\,2i\int_{\mathbb{R}} r_{0}\varphi'\varphi|v|^{2}d\,x=0,
\end{align*}
equality \eqref{eq_MA28} allows us to write condition \eqref{eq_MA22} in the required form
\begin{equation*}
  \mathrm{Re}\,\left(\varphi v, \mathrm{L}_{00}^{+}\varphi v\right)_{L^{2}(\mathbb{R})}= \int_{\mathbb{R}}(\varphi')^{2}|v|^{2}d\,x+2\int_{\mathbb{R}}r_{1}\varphi'\varphi|v|^{2}d\,x \geq \int_{\mathbb{R}} \varphi^{2}|v|^{2}d\,x.
\end{equation*}

\end{proof}

Let us consider the following sequence of intervals $\{\Delta_{n}\}_{n\in \mathbb{Z}}$:
\begin{equation*}
  \Delta_{0}:=[-1;1],\quad \Delta_{n}:=[n;n+1], \quad \Delta_{-n}:=[-n-1;-n]  \quad\text{for each}\quad n\in \mathbb{N}.
\end{equation*}
Let $\{\varphi_{n}\}_{n\in \mathbf{N}}$ be a sequence of real-valued functions from  $W_{2,comp}^{2}(\mathbb{R})$ that satisfy the following conditions:
\begin{itemize}
 \item [($a_{1}$)] $0\leq\varphi_{n}(x)\leq 1$, $x\in \mathbb{R}$;
 \item [($a_{2}$)] $\varphi_{n}(x)=1$, $x\in [-n;n]$;
 \item [($a_{3}$)] $\mathrm{supp}\,\varphi_{n}\subset [-n-1;n+1]$;
 \item [($a_{4}$)] $\varphi_{n}'(x)\geq 0$, $x\in \Delta_{-n}$ and $\varphi_{n}'(x)\leq 0$, $x\in \Delta_{n}$;
 \item [($a_{5}$)] there exists a $K>0$ such that
  \begin{equation*}
   |\varphi_{n}'(x)|\leq K \quad\text{for every}\quad x\in \mathbb{R}.
 \end{equation*}
\end{itemize}

We now substitute the functions $\varphi_{n}$ into \eqref{eq_MA_MainWorkInq} and note that
\begin{equation*}
  \mathrm{supp}\,\varphi_{n}'\subset \Delta_{-n}\cup \Delta_{n}, \quad n\in \mathbb{N}.
\end{equation*}
We then obtain
\begin{align*}
 \int\limits_{-n}^{n} |v|^{2}d\,x & \leq \int_{\mathbb{R}} \varphi_{n}^{2}|v|^{2}d\,x\leq \int_{\mathbb{R}}(\varphi_{n}')^{2}|v|^{2}d\,x +2\int_{\mathbb{R}}r_{1}\varphi_{n}'\varphi_{n}|v|^{2}d\,x \\
  & = \int\limits_{\Delta_{-n}\cup \Delta_{n}}(\varphi_{n}')^{2}|v|^{2}d\,x+2\int\limits_{\Delta_{-n}}r_{1}\varphi_{n}'\varphi_{n}|v|^{2}d\,x +2\int\limits_{\Delta_{n}}r_{1}\varphi_{n}'\varphi_{n}|v|^{2}d\,x \\
  & \leq K^{2}\int\limits_{\Delta_{-n}\cup \Delta_{n}}|v|^{2}d\,x+2\int\limits_{\Delta_{-n}}(r_{1}^{+}-r_{1}^{-})\varphi_{n}'\varphi_{n}|v|^{2}d\,x
   +2\int\limits_{\Delta_{n}}(r_{1}^{+}-r_{1}^{-})\varphi_{n}'\varphi_{n}|v|^{2}d\,x \\
  &  = K^{2}\int\limits_{\Delta_{-n}\cup \Delta_{n}}|v|^{2}d\,x +2\int\limits_{\Delta_{-n}}r_{1}^{+}\varphi_{n}'\varphi_{n}|v|^{2}d\,x-2\int\limits_{\Delta_{n}}r_{1}^{-}\varphi_{n}'\varphi_{n}|v|^{2}d\,x \\
  & -2\int\limits_{\Delta_{-n}}r_{1}^{-}\varphi_{n}'\varphi_{n}|v|^{2}d\,x+2\int\limits_{\Delta_{n}}r_{1}^{+}\varphi_{n}'\varphi_{n}|v|^{2}d\,x \\
  & \leq K^{2}\int\limits_{\Delta_{-n}\cup \Delta_{n}}|v|^{2}d\,x +2\int\limits_{\Delta_{-n}}r_{1}^{+}\varphi_{n}'\varphi_{n}|v|^{2}d\,x-2\int\limits_{\Delta_{n}}r_{1}^{-}\varphi_{n}'\varphi_{n}|v|^{2}d\,x.
\end{align*}
Here, we used
\begin{equation*}
    -2\int\limits_{\Delta_{-n}}r_{1}^{-}\varphi_{n}'\varphi_{n}|v|^{2}d\,x  \leq 0 \quad \text{and} \quad
    2\int\limits_{\Delta_{n}}r_{1}^{+}\varphi_{n}'\varphi_{n}|v|^{2}d\,x  \leq 0.
\end{equation*}
Thus,
\begin{equation}\label{eq_MA_MainWorkInq_1}
   \int\limits_{-n}^{n} |v|^{2}d\,x \leq  K^{2}\int\limits_{\Delta_{-n}\cup \Delta_{n}}|v|^{2}d\,x +2\int\limits_{\Delta_{-n}}r_{1}^{+}\varphi_{n}'|v|^{2}d\,x-2\int\limits_{\Delta_{n}}r_{1}^{-}\varphi_{n}'|v|^{2}d\,x.
\end{equation}


According to assumptions ($\mathrm{m}_{1}$) of Theorem~\ref{th_MAccretivityMAIN_A}, there exist numbers $N_{0} \in \mathbb{N}$ and $C>0$ such that
\begin{align}
  r_{1}^{+}(x) & \leq Cm(x) \quad\text{for every}\quad x\in (-\infty;-N_{0}], \label{eq_MA30_1} \\
  r_{1}^{-}(x) & \leq Cm(x) \quad\text{for every}\quad x\in [N_{0};\infty). \label{eq_MA30_2}
\end{align}
Then, taking into account estimates \eqref{eq_MA30_1} and \eqref{eq_MA30_2}, for $n\geq N_{0}$, we deduce from inequality \eqref{eq_MA_MainWorkInq_1}
\begin{align}\label{eq_MA_MainWorkInq_2}
  \int\limits_{-n}^{n} |v|^{2}d\,x & \leq  K^{2}\int\limits_{\Delta_{-n}\cup \Delta_{n}}|v|^{2}d\,x +2C\int\limits_{\Delta_{-n}}m\varphi_{n}'|v|^{2}d\,x-2C\int\limits_{\Delta_{n}}m\varphi_{n}'|v|^{2}d\,x.
\end{align}

Let us consider the function
\begin{equation*}
      \rho(x):=
  \left\{
    \begin{array}{rl}
      -\int\limits_{x}^{0}\frac{d\,s}{m(s)}, & x\in (-\infty;0], \\
      \int\limits_{0}^{x}\frac{d\,s}{m(s)}, & x\in [0;\infty);
    \end{array}
  \right.
\end{equation*}
which satisfies
\begin{equation*}
  \rho(x)\rightarrow -\infty \quad\text{as}\quad x\rightarrow -\infty, \quad\text{and}\quad \rho(x)\rightarrow +\infty \quad\text{as}\quad x\rightarrow +\infty,
\end{equation*}
in view of condition $(\mathrm{m}_{2})$ of Theorem~\ref{th_MAccretivityMAIN_A}. By condition $(\mathrm{m}_{1})$ of Theorem~\ref{th_MAccretivityMAIN_A}, the derivative has the properties
\begin{equation*}
\rho'(x)=\frac{1}{m(x)}>0\quad\text{for any}\quad x\in \mathbb{R},\quad\text{and}\quad \rho'\in W_{2,loc}^{1}(\mathbb{R}).
\end{equation*}
Then
\begin{equation}\label{eq_ChDer}
  \left(\varphi_{n}[\rho(x)]\right)'=\left(\varphi_{n}[\rho]\right)_{\rho}'(\rho(x))_{x}' =\left(\varphi_{n}[\rho]\right)_{\rho}'\frac{1}{m(x)},
\end{equation}
and
\begin{equation*}
  \varphi_{n}[\rho]\in W_{2,comp}^{2}(\mathbb{R}) \quad\text{for each}\quad n\in \mathbb{Z}.
\end{equation*}
Let us put
\begin{align*}
  \mathrm{X}_{n}^{-} & :=\{x\in \mathbb{R}  \left|\,\rho(x)\in\Delta_{-n}\right.\} \quad\text{for each}\quad n=0,1,2,\dots\,, \\
  \mathrm{X}_{n}^{+} & :=\{x\in \mathbb{R}  \left|\,\rho(x)\in \Delta_{n}\right.\} \quad\text{for each}\quad n=0,1,2,\dots\,,
\end{align*}
and
\begin{equation*}
   \mathrm{U}_{n}  :=\bigcup_{k=0}^{n-1}\mathrm{X}_{n}^{-}\cup\mathrm{X}_{n}^{+}.
\end{equation*}
Note that under condition $(\mathrm{m}_{2})$ of Theorem~\ref{th_MAccretivityMAIN_A}, the following relations hold:
\begin{align}
  \lim_{n\rightarrow-\infty}\sup\mathrm{X}_{n}^{-} & =-\infty, \label{eq_rho_1} \\
  \lim_{n\rightarrow\infty}\inf\mathrm{X}_{n}^{+} & =\infty, \label{eq_rho_2}
\end{align}
and
\begin{equation}\label{eq_rho_3}
   \lim_{n\rightarrow\infty}\mathrm{U}_{n} =\lim_{n\rightarrow\infty}\bigcup_{k=0}^{n-1}\mathrm{X}_{n}^{-}\cup\mathrm{X}_{n}^{+}=\mathbb{R}.
\end{equation}


Let us substitute the functions $\varphi_{n}(\rho(\cdot))$ with $n\in \mathbb{Z}$ into \eqref{eq_MA_MainWorkInq}. As when proving inequalities \eqref{eq_MA_MainWorkInq_1} and \eqref{eq_MA_MainWorkInq_2}, we infer by \eqref{eq_ChDer} with $n\geq N_{0}$ that
\begin{align*}
\int\limits_{\mathrm{U}_{n}} |v(x)|^{2}d\,x
& \leq  K^{2}\int\limits_{\mathrm{X}_{n}^{-}\cup\mathrm{X}_{n}^{+}}|v(x)|^{2}d\,x
+2\int\limits_{\mathrm{X}_{n}^{-}}m(x)(\varphi_{n}[\rho(x)])'|v(x)|^{2}d\,x
 -2\int\limits_{\mathrm{X}_{n}^{+}}m(x)(\varphi_{n}[\rho(x)])'|v(x)|^{2}d\,x \\
& \leq  K^{2}\int\limits_{\mathrm{X}_{n}^{-}\cup\mathrm{X}_{n}^{+}}|v(x)|^{2}d\,x
 +2C\int\limits_{\mathrm{X}_{n}^{-}}m(x)(\varphi_{n}[\rho])_{\rho}'\rho'(x)|v(x)|^{2}d\,x
 -2C\int\limits_{\mathrm{X}_{n}^{+}}m(x)(\varphi_{n}[\rho])_{\rho}'\rho'(x)|v(x)|^{2}d\,x \\
& \leq  K^{2}\int\limits_{\mathrm{X}_{n}^{-}\cup\mathrm{X}_{n}^{+}}|v(x)|^{2}d\,x
 +2CK\int\limits_{\mathrm{X}_{n}^{-}}|v(x)|^{2}d\,x  +2CK\int\limits_{\mathrm{X}_{n}^{+}}|v(x)|^{2}d\,x.
\end{align*}
Thus,
\begin{equation}\label{eq_MA_MainWorkInq_3}
  \int\limits_{\mathrm{U}_{n}} |v(x)|^{2}d\,x \leq  K^{2}\int\limits_{\mathrm{X}_{n}^{-}\cup\mathrm{X}_{n}^{+}}|v(x)|^{2}d\,x
  +2CK\int\limits_{\mathrm{X}_{n}^{-}\cup \mathrm{X}_{n}^{+}}|v(x)|^{2}d\,x \quad\text{whenever}\quad n\geq N_{0}.
\end{equation}
Let us now pass to the limit as $n\rightarrow\infty$ in inequality~\eqref{eq_MA_MainWorkInq_3}. Since $v\in L^{2}(\mathbb{R})$ and because relations \eqref{eq_rho_1} -- \eqref{eq_rho_3} hold, we obtain
\begin{equation*}
  \int\limits_{\mathrm{X}_{n}^{-}\cup \mathrm{X}_{n}^{+}}|v(x)|^{2}d\,x \rightarrow 0 \quad\text{as}\quad n\rightarrow\infty,
\end{equation*}
and, finally, conclude that
\begin{equation*}
  \int\limits_{\mathbb{R}}|v(x)|^{2}d\,x\leq 0, \qquad\text{i.e.,}\qquad v(x)\equiv 0.
\end{equation*}


The proof is complete.
\end{proof}

\section{Interval-type theorem}\label{sec:IntTpTh}
Since functions of class $W_{2,loc}^{1}(\mathbb{R})$ are continuous on the entire axis, then they are locally bounded. Therefore, it follows from condition $(\mathrm{m}_{1})$ of  Theorem~\ref{th_MAccretivityMAIN_A} that the function $r_{1}$ belongs to $L_{loc}^{\infty}(\mathbb{R})$ and, unlike the function $r_{0}$, cannot have local singularities. To include a wider class of admissible functions $r_{1}$, we now present an alternative $m$-accretivity criterion for minimal operators. These conditions are imposed only on the behaviour of the function $r_{1}$. They concern only sequences of intervals in neighborhoods of $\pm\infty$, and there are no additional conditions outside these intervals.

Self-adjointness conditions of this type for Schr\"{o}dinger operators of form \eqref{fde_Shr} with $r\equiv 0$ and with a real continuous potential $q$ on $\mathbb{R}$ were first established in Ismagilov's paper \cite{Ismgl1963}. They have the following form. Let $\Delta_{n}$ be a sequence of disjoint intervals of length $|\Delta_{n}|$ with $n\in \mathbb{Z}$, and let $\Delta_{n}$ tend to $\pm\infty$ as $n\rightarrow\pm\infty$. If
\begin{equation*}
q(x)\geq -\gamma_{n}\;(\gamma_{n}>0)\quad\text{for every}\quad x\in \Delta_{n} \quad\text{and every}\quad n\in \mathbb{Z},
\end{equation*}
and if
\begin{equation*}
\Sigma_{n=-\infty}^{-1}(\gamma_{n}+|\Delta_{n}|^{-2})^{-1}=\Sigma_{n=1}^{\infty}(\gamma_{n}+|\Delta_{n}|^{-2})^{-1}=\infty,
\end{equation*}
then the minimal symmetric operator $\mathrm{L}_{0}$ is self-adjoint.

These results were further developed in \cite{Orchk1974} and the publications cited therein.

The main result of this section is the following statement.
\begin{theoremEng}\label{th_MAccretivityMAIN_B}
Let the coefficients of the formal differential Schr\"{o}dinger expression \eqref{fde_Shr} satisfy conditions \eqref{cond_Main}, and let the pre-minimal operators $\mathrm{L}_{00}$ and $\mathrm{L}_{00}^{+}$ be accretive. If there exists a sequence of intervals $\Delta_{n}:=[a_{n};b_{n}]$ with $n\in \mathbb{Z}$ such that
\begin{equation*}
  -\infty< a_{n}<b_{n}<\infty,
  \quad b_{n}\rightarrow -\infty\quad\text{as}\quad n\rightarrow -\infty,\quad\text{and}\quad a_{n}\rightarrow \infty\quad\text{as}\quad n\rightarrow\infty,
\end{equation*}
and if
\begin{enumerate}
  \item [($\mathrm{B}_{1}$)] there exists a number $\delta>0$ such that $|\Delta_{n}|\geq \delta$ for each $n\in \mathbb{N}$,
  \item [($\mathrm{B}_{2}$)] there exists a number $C>0$ such that
  \begin{align*}
    r_{1}^{+}(x) & \leq C|\Delta_{-n}|\quad\text{for every}\quad x\in \Delta_{-n},\quad\text{whatever}\quad n\in \mathbb{N},  \\
    r_{1}^{-}(x) & \leq C|\Delta_{n}|\quad\text{for every}\quad x\in \Delta_{n},\quad\text{whatever}\quad n\in \mathbb{N};
  \end{align*}
\end{enumerate}
then the minimal operators $\mathrm{L}_{0}$ and $\mathrm{L}_{0}^{+}$ are $m$-accretive, i. e., $\mathrm{L}_{0}=\mathrm{L}$ and $\mathrm{L}_{0}^{+}=\mathrm{L}^{+}$.
\end{theoremEng}

\begin{remarkEng}\label{cor_ThB_10}
It follows from Theorem~\ref{th_MAccretivityMAIN_B} that if the assumptions given therein are fulfilled for the function $r_{1}$ and if
\begin{equation*}
  (\mathrm{L}_{00}u,u)\in \Lambda\quad\text{for every}\quad u \in \mathrm{Dom}(\mathrm{L}_{00}), \quad\text{and}\quad (\mathrm{L}_{00}^{+}v,v)\in \overline{\Lambda}\quad\text{for every}\quad\quad v \in \mathrm{Dom}(\mathrm{L}_{00}^{+}),
\end{equation*}
then $\rho(\mathrm{L}_{0})\supset\mathbb{C}\setminus \Lambda$ and $\rho(\mathrm{L}_{0}^{+})\supset\mathbb{C}\setminus \overline{\Lambda}$, where $\Lambda$ is a closed sector of the half-plane $\mathrm{Re}\,\lambda \geq 0$ with vertex at the point~$0$.
\end{remarkEng}


\begin{proof}[Proof of Theorem~\ref{th_MAccretivityMAIN_B}]
As in the proof of Theorem~\ref{th_MAccretivityMAIN_A}, let $v\in L^{2}(\mathbb{R})$ be a solution of the equation $\mathrm{L}^{+}v=0$. We show that $v\equiv 0$. Without loss of generality, we assume that the intervals $\Delta_{n}=[a_{n};b_{n}]$ with $n\in \mathbb{Z}$ are mutually disjoint, and that $b_{-1}<0$ and $a_{1}>0$. Let $\{\varphi_{n}\}_{n\in \mathbf{N}}$ be a sequence of real-valued compactly supported functions from $W_{2,comp}^{2}(\mathbb{R})$ that satisfy the following conditions:
\begin{itemize}
 \item [($a_{1}$)] $0\leq\varphi_{n}(x)\leq 1$ for every $x\in \mathbb{R}$;
 \item [($a_{2}$)] $\varphi_{n}(x)=1$ for every $x\in [b_{-n};a_{n}]$;
 \item [($a_{3}$)] $\mathrm{supp}\,\varphi_{n}\subset [a_{-n};b_{n}]$;
 \item [($a_{4}$)] $\varphi_{n}'(x)\geq 0$ for any $x\in [a_{-n};b_{-n}]$, and $\varphi_{n}'(x)\leq 0$ for any $x\in [a_{n};b_{n}]$;
 \item [($a_{5}$)] for some number $K>0$ the inequalities hold
  \begin{equation*}
   |\varphi_{n}'(x)|\leq
\begin{cases}
K|\Delta_{-n}|^{-1}\quad & \text{if}\quad x<0, \\
K|\Delta_{n}|^{-1}\quad & \text{if}\quad x>0.
\end{cases}
 \end{equation*}
\end{itemize}
Then, by \eqref{eq_MA_MainWorkInq} and by assumptions ($\mathrm{B}_{1}$) and ($\mathrm{B}_{2}$) of the theorem, and since $\mathrm{supp}\,\varphi_{n}'\subset \Delta_{-n}\cup \Delta_{n}$ for each $n\in \mathbb{N}$, we obtain the following:
\begin{align*}
 \int_{b_{-n}}^{a_{n}} |v|^{2}d\,x & \leq \int_{\mathbb{R}} \varphi_{n}^{2}|v|^{2}d\,x\leq \int_{\mathbb{R}}(\varphi_{n}')^{2}|v|^{2}d\,x +2\int_{\mathbb{R}}r_{1}\varphi_{n}'\varphi_{n}|v|^{2}d\,x \\
  & = \int\limits_{\Delta_{-n}\cup \Delta_{n}}(\varphi_{n}')^{2}|v|^{2}d\,x+2\int\limits_{\Delta_{-n}}r_{1}\varphi_{n}'\varphi_{n}|v|^{2}d\,x +2\int\limits_{\Delta_{n}}r_{1}\varphi_{n}'\varphi_{n}|v|^{2}d\,x \\
  & \leq K^{2}\delta^{2}\int\limits_{\Delta_{-n}\cup \Delta_{n}}|v|^{2}d\,x+2\int\limits_{\Delta_{-n}}(r_{1}^{+}-r_{1}^{-})\varphi_{n}'\varphi_{n}|v|^{2}d\,x
   +2\int\limits_{\Delta_{n}}(r_{1}^{+}-r_{1}^{-})\varphi_{n}'\varphi_{n}|v|^{2}d\,x \\
  &  = K^{2}\delta^{2}\int\limits_{\Delta_{-n}\cup \Delta_{n}}|v|^{2}d\,x +2\int\limits_{\Delta_{-n}}r_{1}^{+}\varphi_{n}'\varphi_{n}|v|^{2}d\,x-2\int\limits_{\Delta_{n}}r_{1}^{-}\varphi_{n}'\varphi_{n}|v|^{2}d\,x \\
  & -2\int\limits_{\Delta_{-n}}r_{1}^{-}\varphi_{n}'\varphi_{n}|v|^{2}d\,x+2\int\limits_{\Delta_{n}}r_{1}^{+}\varphi_{n}'\varphi_{n}|v|^{2}d\,x \\
  & \leq K^{2}\delta^{2}\int\limits_{\Delta_{-n}\cup \Delta_{n}}|v|^{2}d\,x +2\int\limits_{\Delta_{-n}}r_{1}^{+}\varphi_{n}'\varphi_{n}|v|^{2}d\,x-2\int\limits_{\Delta_{n}}r_{1}^{-}\varphi_{n}'\varphi_{n}|v|^{2}d\,x \\
  & \leq K^{2}\delta^{2}\int\limits_{\Delta_{-n}\cup \Delta_{n}}|v|^{2}d\,x+2CK\int\limits_{\Delta_{-n}}|v|^{2}d\,x+2CK\int\limits_{\Delta_{n}}|v|^{2}d\,x \\
  & = K^{2}\delta^{2}\int\limits_{\Delta_{-n}\cup \Delta_{n}}|v|^{2}d\,x+2CK\int\limits_{\Delta_{-n}\cup \Delta_{n}}|v|^{2}d\,x.
\end{align*}
Here, we use the inequalities
\begin{equation*}
    -2\int\limits_{\Delta_{-n}}r_{1}^{-}\varphi_{n}'\varphi_{n}|v|^{2}d\,x  \leq 0 \quad\text{and}\quad
    2\int\limits_{\Delta_{n}}r_{1}^{+}\varphi_{n}'\varphi_{n}|v|^{2}d\,x  \leq 0.
\end{equation*}
Therefore, we have
\begin{align}\label{eq_MA50}
  \int_{a_{-n}}^{b_{n}} |v|^{2}d\,x & \leq K^{2}\delta^{2}\int\limits_{\Delta_{-n}\cup \Delta_{n}}|v|^{2}d\,x +2CK\int\limits_{\Delta_{-n}\cup \Delta_{n}}|v|^{2}d\,x \quad\text{for each}\quad n\in \mathbb{N}.
\end{align}
Passing in \eqref{eq_MA50} to the limit as $n\rightarrow\infty$, we infer that $v\equiv 0$ because $v\in L^{2}(\mathbb{R})$.

The proof is complete.

\end{proof}

\begin{appendices}



\section{Appendix}\label{sec:AppA}

We present some additional results on accretive differential operators in the Hilbert space $L^{2}(\mathbb{R})$
\begin{equation*}
  \mathrm{L}_{00}u=-u''+au'+bu, \qquad u\in C_{comp}^{\infty}(\mathbb{R})
\end{equation*}
with regular complex-valued coefficients $a\in W_{2,loc}^{1}(\mathbb{R})$ and $b\in L_{loc}^{2}(\mathbb{R})$. Theorems \ref{th_MAccretivityMAIN_A} and \ref{th_MAccretivityMAIN_B} of this paper are new for such operators as well.

The following theorem gives necessary and sufficient conditions for the accretivity of the operator $\mathrm{L}_{0}$, as well as the operator $\mathrm{L}_{00}^{+}$, which is defined on smooth functions with compact support. It is a special case of Theorem 2.3 of \cite{MzVrbActaMathSin2019}, which studies a wider class of operators.
\begin{theoremEng}\label{th4}
The following statements hold: \\
\noindent (i) If there exists a function $f\in L_{loc}^{2}(\mathbb{R})$ such that
\begin{equation}\label{eq_A10}
  \mathrm{Re}\,b-\frac{1}{2}\mathrm{Re}\,a'+\frac{1}{4}(\mathrm{Im}\,a)^{2}\geq f'+f^{2} \quad\text{in}\quad \mathcal{D}'(\mathbb{R}),
\end{equation}
then the operator $\mathrm{L}_{00}$ is accretive.\\
\noindent (ii) On the contrary, if the operator $\mathrm{L}_{00}$ is accretive, then there exists a function $f\in L_{loc}^{2}(\mathbb{R})$ such that inequality~\ref{eq_A10} holds.
\end{theoremEng}

\begin{corollaryEng}\label{cr_th4}
The operator $\mathrm{L}_{00}$ is accretive if
\begin{equation*}
  \mathrm{Re}\,b-\frac{1}{2}\mathrm{Re}\,a'-\frac{1}{4}(\mathrm{Im}\,a)^{2}\geq 0 \quad\text{a. e}.
\end{equation*}
\end{corollaryEng}
This corollary admits a simple direct proof.

Because the coefficients are distributions, classical derivatives may not exist. In general, $\operatorname{Dom}(L_{00}) \cap C^1(\mathbb{R}) = \{0\}$, which necessitates the use of quasi-derivatives as defined in Section~\ref{sec:SO}.


As established by the authors in \cite{MkhMlb2013}, if $a=0$ and the operator $\mathrm{L}_{00}$ is accretive, then its closure $\mathrm{L}_{0}=\mathrm{L}$ is an $m$-accretive operator. The following statement shows that even in the case of smooth real-valued coefficients $a$, $b$ and $a\neq 0$ this is, in general, not the case.

\begin{theoremEng}\label{th5}
Let $v\in L^{2}(\mathbb{R})$ be a positive smooth function. Then there exist smooth real-valued coefficients $a$ and $b$ such that \\
\noindent (i) the operator $\mathrm{L}_{00}$ (and hence $\mathrm{L}_{0}$) is accretive; \\

\noindent (ii) the operator $\mathrm{L}_{0}$ is not $m$-accretive, because $\mathrm{L}_{0}^{\ast}v+v=0$.
\end{theoremEng}
\begin{proof}
Let the function $v$ satisfy the conditions of the theorem. If these functions are smooth and real-valued, then
\begin{equation*}
 \mathrm{Re}(\mathrm{L}_{00}u,u)=\int\limits_{\mathbb{R}}|u'|^{2}dx+\int\limits_{\mathbb{R}}\left(b-\frac{a'}{2}\right)|u|^{2}dx, \qquad u\in C_{comp}^{\infty}(\mathbb{R}).
\end{equation*}
Therefore, for the operator $\mathrm{L}_{00}$ to be accretive, it is sufficient to satisfy the inequality
\begin{equation*}
  b-\frac{a'}{2}=w\geqslant 0 \quad\text{a. e.}
\end{equation*}
Then the adjoint operator
\begin{equation*}
  \mathrm{L}_{00}^{\ast}=\mathrm{L}_{0}^{\ast}=-\frac{d^{2}}{dx^{2}}-a\frac{d}{dx}+(b-a').
\end{equation*}
Therefore, the equation $\mathrm{L}_{0}^{\ast}v+v=0$ is equivalent to the equality
\begin{equation*}
  -v''-av'+(b-a'+1)v=0.
\end{equation*}
This is a first-order linear differential equation with respect to the unknown function $a$.
Substituting here $b=w+\frac{a'}{2}$, where the function $w$ is smooth and non-negative, we have
\begin{equation*}
  -v''-av'+(w-\frac{a'}{2}+1)v=0,
\end{equation*}
or
\begin{equation*}
  \frac{a'}{2}+\frac{v'}{v}a=w+1-\frac{v''}{v}.
\end{equation*}
Hence, we have that
\begin{equation*}
  a(x)=v^{-2}(x)\left(C+2\int\limits_{0}^{x}(w(t)+1)v^{2}(t)-v(t)v''(t)dt\right),
\end{equation*}
where $C$ is an arbitrary real constant. Then $b=w+\frac{a'}{2}$.

Since $\mathrm{L}_{0}^{\ast}v=-v\neq 0$, the point $-1$ belongs to the point spectrum of $\mathrm{L}_{0}^{\ast}$, hence $\mathrm{L}_{0}$ is not $m$-accretive.
\end{proof}

\vspace{25pt}

\section{Appendix}\label{sec:AppB}

Theorem \ref{th_ApendixB} of this appendix demonstrates that the integral condition $\int (1/m) = \infty$ in Theorem \ref{th_MAccretivityMAIN_A} is sharp. We construct a class of smooth, real-valued coefficients of the operators where $\int (1/m) < \infty$ and show that these operators fail to be $m$-accretive.

\begin{theoremEng}\label{th_ApendixB}
Let the non-negative function $a\in C^{1}(\mathbb{R})$ satisfy the conditions:
\begin{itemize}
  \item [(i)] $\frac{a'}{a^{2}}\rightarrow 0$, $|x|\rightarrow \infty$;
  \item [(ii)] $\frac{a'}{a^{2}}$ has bounded variation on the set $\mathbb{R}\setminus(-c;c)$, $c>0$;
      \item[(iii)] $a \rightarrow \infty$, $|x|\to\infty$;
  \item [\rm(iv)] $a^{-1}\in L^{1}(\mathbb{R})$.
\end{itemize}
Closure $\mathrm{L}_0$ of the differential operator
\begin{equation*}
  \mathrm{L}_{00}u=-u''+a(x)u'+b(x)u, \quad b(x)=\frac{1}{2}a'(x),\quad u\in C_{comp}^{\infty}(\mathbb{R})
\end{equation*}
is an accretive but not $m$-accretive operator in a complex Hilbert space $L^{2}(\mathbb{R})$.
\end{theoremEng}
\begin{proof}
The accretiveness of the operator follows from the corollary~\ref{cr_th4} of Theorem~\ref{th4} and the equality $b(x)=\frac{1}{2}a'(x)$. We prove that under assumptions (i), (ii), (iii) and (iv) of the theorem it is not $m$-accretive. By the Lumer--Phillips theorem this is equivalent to the fact that $\mathrm{Ker}\left(\mathrm{L}_{0}^{\ast}+\mathrm{I}\right)\neq \{0\}$.

Standard integration by parts gives:
\begin{equation*}
  \mathrm{L}_{0}^{\ast}v=-v''-(av)'+bv=-v''-av'-\frac{1}{2}a'v.
\end{equation*}
Consider the equation $\left(\mathrm{L}_{0}^{\ast}+\mathrm{I}\right)v=0$:
\begin{equation}\label{eq_B10}
  -v''-av'+\left(1-\frac{1}{2}a'\right)v=0.
\end{equation}
We make a substitution
\begin{equation}\label{eq_B20}
  v{x}=e^{-A(x)/2}\varphi, \quad A(x)=\int_{0}^{x}a(t)d\,t.
\end{equation}
We calculate
\begin{align*}
  v' & =e^{-A/2}\left(-\frac{a}{2}\varphi+\varphi'\right) \\
  v'' & =e^{-A/2}\left[\left(\frac{a^{2}}{4}-\frac{a'}{2}\right)\varphi-a\varphi'+\varphi''\right].
\end{align*}
Substituting these equalities into \eqref{eq_B10} and reducing by $e^{-A/2}\neq 0$, we obtain the differential equation
\begin{equation}\label{eq_B30}
  -\varphi''+W(x)\varphi=0, \qquad W(x):=\frac{a^{2}}{4}+1.
\end{equation}
At the same time
\begin{equation}\label{eq_B40}
  v\in L^{2}(\mathbb{R}) \Leftrightarrow \int_{-\infty}^{\infty}e^{-A(x)}|\varphi(x)|^{2}d\,x<\infty.
\end{equation}
Equation~\eqref{eq_B30} has the form $\varphi''=W(x)\varphi$ with $W(x)>0$ if $|x|\gg 1$. The function $W(x)\rightarrow +\infty$ at $|x|\rightarrow\infty$. To apply Olver's theorem \cite[Ch.~6, Theorem~11.1]{OlvrASF1997} (see also \cite{Olvr1961}) we must verify its conditions. Lemma \ref{lm_AppendixB} guarantees these conditions. Consequently, the fundamental system of solutions $\{\varphi_+, \varphi_-\}$ has the asymptotic form
\begin{equation}\label{eq_B50}
  \varphi_{\pm}(x)=W^{-1/4}(x)\exp\left(\pm \int_{c}^{x}W^{1/2}(t)d\,t\right)(1+\eta_{\pm}(x)),
\end{equation}
where $\eta_{\pm}(x)\rightarrow 0$ at $x\rightarrow +\infty$ if the following conditions are holds:
\begin{align}
  & \frac{W'(x)}{W^{3/2}(x)}\rightarrow 0 \quad\text{as}\quad x\rightarrow +\infty\quad\text{and}\quad  \label{eq_B60}  \\
  & \text{the function }\frac{W'}{W^{3/2}}\text{has bounded variation on}\quad  [c;\infty). \label{eq_B70}
\end{align}
According to our Lemma~\ref{lm_AppendixB}, conditions \eqref{eq_B60}, \eqref{eq_B70} are satisfied if the conditions of Theorem~\ref{th_ApendixB} are fulfilled.

Besides
\begin{equation*}
  \sqrt{W(x)}=\sqrt{1+\frac{a^{2}(x)}{4}}=\frac{a(x)}{2}\sqrt{1+\frac{4}{a^{2}(x)}}=\frac{a(x)}{2} +\frac{1}{a(x)}+O\left(\frac{1}{a^{3}(x)}\right), \quad |x|\rightarrow \infty.
\end{equation*}
Since $a^{-1}\in L^{1}(\mathbb{R})$, and $a(x)\rightarrow +\infty$, we also have
\begin{equation*}
  \frac{1}{a^{3}(x)} \in L^{1}(\mathbb{R}).
\end{equation*}
Therefore, there exist constants $C_{\pm}$ such that
\begin{equation}\label{eq_B80}
  \int_{0}^{x} \sqrt{W(t)}d\,t=\frac{A(x)}{2}+C_{\pm}+o(1), \quad |x|\rightarrow \infty.
\end{equation}
Let us investigate the $L^{2}$-integrability of the solutions of the differential equation \eqref{eq_B10} on the right half-axis, based on the formula \eqref{eq_B50}. Returning to the function $v$ by the formula \eqref{eq_B20}, we have
\begin{equation*}
  v_{\pm}(x)=e^{-A(x/2)}\varphi_{\pm}.
\end{equation*}
For the decaying solution, it follows from the formulas \eqref{eq_B40}, \eqref{eq_B50} that
\begin{equation*}
  |v_{-}(x)|^{2}\asymp e^{-A(x)}W^{-1/2}(x)\exp\left(-2\int_{0}^{x} \sqrt{W(t)}d\,t\right)\asymp W^{-1/2}(x)e^{-2A(x)}.
\end{equation*}
This expression approaches zero faster than exponentially.

For an increasing solution, we have
\begin{equation*}
  |v_{+}(x)|^{2}\asymp e^{-A(x)}W^{-1/2}(x)\exp\left(2\int_{0}^{x} \sqrt{W(t)}d\,t\right)\asymp W^{-1/2}(x)e^{2C_{+}}.
\end{equation*}
Since
\begin{equation*}
  W^{-1/2}(x)\asymp \frac{2}{a(x)},
\end{equation*}
we have
\begin{equation*}
  |v_{+}(x)|^{2}\leq \frac{c}{a(x)}.
\end{equation*}
From which, by virtue of condition (iii) of the theorem, it follows that
\begin{equation*}
  v_{+}\in L^{2}(d;+\infty).
\end{equation*}
Therefore, both fundamental solutions belong to the space $L^{2}$ on the right half-axis (the Weyl limit-circle case at $+\infty$).

Let us investigate the $L^{2}$-integrability of the solutions on the left half-axis $(-\infty;0)$.

Let us recall the connection between the solutions of the equation $(\mathrm{L}_{0}^{\ast}+\mathrm{I})v=0$ and \eqref{eq_B30}:
\begin{equation*}
  v{x}=e^{-A(x)/2}\varphi, \quad\text{where}\quad A(x)=\int_{0}^{x}a(t)d\,t.
\end{equation*}
Since the function $a$ is non-negative, then for $x<0$ we have $A(x)\leq 0$, and $A(x)\rightarrow -\infty$ for $x\rightarrow -\infty$. The condition $v\in L^{2}(-\infty;0)$ is equivalent to the following
\begin{equation*}
  \int_{-\infty}^{0}e^{-A(x)}|\varphi(x)|^{2}d\,x<\infty.
\end{equation*}
Let us apply Olver's theorem to study the asymptotics as $x\rightarrow -\infty$. For convenience, we will write the integral in the exponent of the WKB approximation from $x$ to $0$. Since
\begin{equation*}
  \sqrt{W(x)}=\frac{a(x)}{2} +\frac{1}{a(x)}+O\left(\frac{1}{a^{3}(x)}\right),
\end{equation*}
we obtain
\begin{equation*}
  \int_{x}^{0} \sqrt{W(t)}d\,t
  =\int_{x}^{0}\left(\frac{a(t)}{2} +\frac{1}{a(t)}+O\left(a^{-3}(t)\right)\right)d\,t
  =-\frac{A(x)}{2}+C_{-}+o(1), \quad x\rightarrow -\infty,
\end{equation*}
where we used the fact that $\int_{x}^{0} a(t)d\,t=-A(x)$, and the integral $\int_{-\infty}^{0} \frac{d\,t}{a(t)}$ converges due to the condition $a^{-1}\in L^{1}(\mathbb{R})$, forming the constant $C_{-}$.

Then the fundamental system of solutions $\varphi_{1}(x)$ and $\varphi_{2}(x)$ at $x\rightarrow -\infty$ has asymptotic
\begin{align*}
\varphi_{1}(x) & \asymp W^{-1/4}(x)\exp\left(\int_{x}^{0} \sqrt{W(t)}d\,t\right)\asymp W^{-1/4}(x) e^{-A(x)/2}e^{C_{-}}, \\
\varphi_{2}(x) & \asymp W^{-1/4}(x)\exp\left(-\int_{x}^{0} \sqrt{W(t)}d\,t\right)\asymp W^{-1/4}(x)e^{A(x)/2}e^{-C_{-}}.
\end{align*}
Let us check whether the corresponding solutions $v_{1}(x)$ and $v_{2}(x)$ belong to the space $L^{2}(-\infty;0)$
\begin{equation*}
  |v_{1}(x)|^{2}=e^{-A(x)}|\varphi_{1}(x)|^{2}\asymp e^{-A(x)}W^{-1/2}(x)e^{-A(x)}=W^{-1/2}(x)e^{-2A(x)}.
\end{equation*}
Since $a(x)\rightarrow +\infty$, then $W^{-1/2}(x)\rightarrow 0$. However, $A(x)\rightarrow -\infty$, so the factor $e^{-A(x)}$ approaches $+\infty$ exponentially. Therefore, the solution is $v_{1}\notin L^{2}(-\infty;0)$.

For the second solution we have as $x\rightarrow -\infty$
\begin{equation*}
  |v_{2}(x)|^{2}=e^{-A(x)}|\varphi_{2}(x)|^{2}\asymp e^{-A(x)}W^{-1/2}(x)e^{A(x)}=W^{-1/2}(x).
\end{equation*}
Since for $x\rightarrow -\infty$ we have $W^{-1/2}(x)\asymp \frac{2}{a(x)}$, and by the condition of the theorem $a^{-1}\in L^{1}(\mathbb{R})$, we obtain that $v_{2}\in L^{2}(-\infty;0)$.

On the left half-axis $(-\infty;0)$ there is only one linearly independent solution of the equation $\left(\mathrm{L}_{0}^{\ast}+\mathrm{I}\right)v=0$ belonging to the space $L^{2}$. This means that the differential operator exhibits the Weyl limit point case at $-\infty$. Since on the right half-axis the case of a limit circle is realized, then on the entire axis $\mathbb{R}$ the equation has exactly one linearly independent $L^{2}$-solution. Thus, $\mathrm{dim}\,\mathrm{Ker}\left(\mathrm{L}_{0}^{\ast}+\mathrm{I}\right)=1$. The proof is   complete.
\end{proof}

Let us now present the technical lemma that we used in proving Theorem \ref{th_ApendixB}.

\begin{lemmaEng}\label{lm_AppendixB} Let the non-negative function $a \in C^2(\mathbb{R})$ satisfy the conditions:
\begin{itemize}
\item[(i)] $\displaystyle \frac{a'}{a^2} \to 0$ as $|x|\to\infty$;
\item[(ii)] the function $\displaystyle \frac{a'}{a^2}$ has bounded variation on the set $\mathbb{R}\setminus(-c,c)$ for some $c>0$;
\item[(iii)] $a\to +\infty$ as $|x|\to\infty$.
\end{itemize}
Then the function $W=\frac14 a^2+1$ satisfies the following properties:
\begin{itemize}
\item[(*)] $\displaystyle \frac{W'}{W^{3/2}}\to 0$ as $|x|\to\infty$;
\item[(**)] the function $\displaystyle \frac{W'}{W^{3/2}}$ has bounded variation on the set $\mathbb{R}\setminus(-c',c')$ for some $c'>0$.
\end{itemize}
\end{lemmaEng}

\begin{proof} Set
\begin{equation}\label{B.8}
\gamma(x):=a^2(x)+4,\qquad f(x):=2\frac{a'(x)}{a^2(x)} =\frac{(a^2(x))'}{a^3(x)}=\frac{\gamma'(x)}{(\gamma(x)-4)^{3/2}}.
\end{equation}
Then $W'=\frac{\gamma'(x)}4$ and
\begin{equation}\label{B.9}
g(x):=\frac{W'(x)}{W^{3/2}(x)}=\frac{\gamma'(x)/4}{(\gamma(x)/4)^{3/2}}
=\frac{2\,\gamma'(x)}{\gamma^{3/2}(x)}.
\end{equation}
Condition (iii) implies the existence of $c_0>0$ such that $a(x)>0$ for all $|x|\ge c_0$. In the following, increasing $c$ as necessary, we assume that $a(x)>0$ and $\gamma(x)>4$ on $\mathbb{R}\setminus(-c,c)$.

From formulas \eqref{B.8} and \eqref{B.9} we have
\begin{equation}\label{B.10}
g(x)=f(x)\,\varphi(x),\qquad \varphi(x)=2\left(1-\frac{4}{\gamma(x)}\right)^{3/2}
=2\left(\frac{a^2(x)}{a^2(x)+4}\right)^{3/2}.
\end{equation}
Since $\gamma(x)\ge 4$ for $a \geq 0$, then $0\le \varphi(x)\le 2$, and $\varphi (x)\rightarrow 2$ for $|x|\to\infty$.

Let us prove the second statement of the theorem. Since $a \in C^2$ and $ a>0$ on $\mathbb{R}\setminus(-c,c)$, the functions $f$ and $g$ belong to the class $C^1$ on this set. Hence we find
\begin{equation}\label{B.11}
f'(x)=\frac{2\gamma''(x)(\gamma(x)-4)-3(\gamma'(x))^2}{2(\gamma(x)-4)^{5/2}},
\qquad
g'(x)=\frac{4\gamma''(x)\delta(x)-6(\gamma'(x))^2}{\gamma^{3/2}(x)}.
\end{equation}

Let us express the numerator $g'$ in terms of the numerator $f'$.

Since
\[
4\gamma''\gamma - 6(\gamma')^2 = \left[ 2\gamma''(\gamma-4) - 3(\gamma')^2 \right] + 16\gamma'' = 4(\gamma-4)^{5/2}f' + 16\gamma'',
\]
we obtain
\begin{equation}\label{B.12}
g' = 4\left(\frac{\gamma-4}{\gamma}\right)^{5/2} f' + \frac{16\gamma''}{\gamma^{3/2}}.
\end{equation}

We rewrite the formula for $\gamma''$ in a convenient form. From the equality $f = 2a'/a^2$ we have $a' = \frac12 f a^2$, whence
\[
a'' = \frac{a^2}{2}(f' + f^2), \quad
\gamma'' = (a^2)'' = 2(a')^2 + 2aa'' = a^3f+\frac32 f^2a^4.
\]

Therefore,
\begin{equation}\label{B.13}
g'(x) = 2\left(\frac{\gamma-4}{\gamma}\right)^{3/2} f'(x) + \frac{12(\gamma-4)^2}{\gamma^{5/2}} f^2(x).
\end{equation}

However, $0 \le 2\left(\frac{\gamma-4}{\gamma}\right)^{3/2} \le 2$, $0 \le \frac{12(\gamma-4)^2}{\gamma^{5/2}} \le 6$.

Hence, from (B.14) we obtain the estimate
\begin{equation}\label{B.14}
|g'(x)| \le 2|f'(x)| + 6f^2(x).
\end{equation}

From conditions (i)--(iii) of the theorem it follows that $f \in L^2(c,\infty)$. Indeed, from the equality $(2/a)' = -f$ we have $f^2 = -f(2/a)'$. Integrating by parts on $[c,R]$:
\begin{equation}\label{B.15}
\int_c^R f^2 dx = -\left[ \frac{2f}{a} \right]_c^R + \int_c^R \frac{2f'}{a} dx = -\frac{2f(R)}{a(R)} + \frac{2f(c)}{a(c)} + 2\int_c^R \frac{f'}{a} dx.
\end{equation}

However, $\frac{f(R)}{a(R)} \to 0$, $R \to +\infty$.

Let $m := \inf_{x \ge c} a(x)$. Since $a$ is continuous, positive on $[c,\infty)$ and tends to $\infty$, we have $m>0$. Therefore
\[
\int_c^\infty \frac{|f'|}{a} dx \le \frac1m \int_c^\infty |f'(x)| = \frac1m \operatorname{Var}_{[c,\infty)}(f) < \infty
\]
by condition (ii). Both components of the right-hand side of (B.16) are finite, and the left-hand side is an integral of a non-negative function, so the integral in the left-hand side of \eqref{B.15} is finite.

Integrating the inequality \eqref{B.14} on $[c,\infty)$:
\[
\operatorname{Var}_{[c,\infty)}(g) = \int_c^\infty |g'(x)| dx \le 2\int_c^\infty |f'(x)| dx + 6\int_c^\infty f^2(x) dx =
\]
\[
= 2\operatorname{Var}_{[c,\infty)}(f) + 6\int_c^\infty f^2(x) dx < \infty.
\]

A similar argument (or replacing $x \mapsto -x$) yields $\operatorname{Var}_{(-\infty,-c]}(g) < \infty$.
The proof is complete.

\end{proof}
We present two examples illustrating our results.
\vspace{0.5cm}

\noindent\textbf{Example 1.} Let $b(x) = \frac{1}{2}a'(x)$, and the coefficient $a$ is given by the formula
\begin{equation*}
    a(x) = C(1+x^2)^\delta, \quad C > 0.
\end{equation*}
Then if $\delta > \frac{1}{2}$, then the condition $a^{-1} \in L^1(\mathbb{R})$ is satisfied, and by Theorem~\ref{th_ApendixB} the operator $\mathrm{L}_0$ is accretive, but not $m$-accretive in the space $L^2(\mathbb{R})$ (the case of a limit circle on $+\infty$). If $\delta \in \left(0, \frac{1}{2}\right]$, then the integral diverges, and by Theorem~\ref{th_MAccretivityMAIN_A} the operator $\mathrm{L}_0$ is $m$-accretive in the space $L^2(\mathbb{R})$.

\vspace{0.5cm}

\noindent\textbf{Example 2.} Let $b(x) = \frac{1}{2}a'(x)$, and the coefficient has a non-power growth:
\begin{equation*}
    a(x) = C\sqrt{1+x^2} \ln^\delta(e+x^2), \quad C>0.
\end{equation*}
Then if $\delta > 1$, the integral $\int_{\mathbb{R}} a^{-1}(x)dx$ converges, so by Theorem~\ref{th_ApendixB} the operator $\mathrm{L}_0$ is accretive, but not $m$-accretive in the space $L^2(\mathbb{R})$. If $\delta \in (0, 1]$, then the corresponding integral diverges, and by Theorem~\ref{th_MAccretivityMAIN_A} the operator $\mathrm{L}_0$ is $m$-accretive in the space $L^2(\mathbb{R})$, since $r=\frac{1}{2i}a$.

These examples confirm that the boundary between accretivity and $m$-accretivity is completely determined by the integrability of $1/m$ at infinity, precisely dictated by Theorem \ref{th_MAccretivityMAIN_A}.
\end{appendices}

\vspace{25pt}

\section*{Acknowledgments}
The work of the first-named author (Vladimir Mikhailets) was funded by the Isaac Newton Institute of Mathematical Sciences "Solidarity Program" and the London Mathematical Society. The author wishes to thank the Department of Mathematics, King's College London, for its  hospitality.

The work of the second-named author (Volodymyr Molyboga) was supported by a grant from the Simons Foundation (SFI-PD-Ukraine-00014586).

The authors are grateful to the reviewers for their valuable remarks, which helped improve the paper.




\begin{thebibliography}{99}

\bibitem{AlbGszHKH2005}
 {Albeverio, S.; Gesztesy, F.; H{ø}egh Krohn, R.; Holden, H.}
 {\textit{Solvable models in quantum mechanics}},
 {2nd ed.},
 {With an appendix by Pavel Exner};
 {AMS Chelsea Publishing}:
 {Providence, RI},
 {2005};
 {xiv+488~pp.};
 doi:~{10.1090/chel/350}.


\bibitem{AlKsMl2010}
 {Albeverio, S.; Kostenko, A.; Malamud, M.}
 {Spectral theory of semibounded {S}turm-{L}iouville operators with local interactions on a discrete set}.
 {\textit{J. Math. Phys.}}
 {\textbf{2010}},
 {51},
 {no.~10},
 {102102},
 {24~pp.};
 doi:~{10.1063/1.3490672}.


\bibitem{ClGs2003}
 {Clark, S.; Gesztesy, F.}
 {On Povzner--Weinholtz-type self-adjointness results for matrix-valued Sturm--Liouville operators}.
 {\textit{Proceedings of the Royal Society of Edinburgh}}
 {\textbf{2003}},
 {133A},
 {no.~4},
 {747--758};
 doi:~{10.1017/S0308210500002651}.


\bibitem{DnfSchw_II_1963}
 {Dunford, N.; Schwartz, J. T.}
 {\textit{Linear operators. Part II: Spectral theory. Selfadjoint operators in Hilbert space}}.
 {With the assistance of William G. Bade and Robert G. Bartle.};
 {Interscience Publishers John Wiley \& Sons}:
 {New York - London},
 {1963};
 {ix+pp. 859--1923+7};
 {\MR{0188745}}.


\bibitem{Glzmn1966}
 {Glazman, I.}
 {\textit{Direct methods of qualitative spectral analysis of singular differential operators}};
 {Daniel Davey \& Co., Inc.}:
 {New York},
 {1966};
 {ix+234~pp.};
 {\MR{0190800}}.


\bibitem{GszNclZnch2024}
 {Gesztesy, F.; Nichols, R.; Zinchenko, M.}
 {\textit{Sturm-{L}iouville operators, their spectral theory, and some  applications}};
 {American Mathematical Society Colloquium Publications},
 {\textbf{67}},
 {American Mathematical Society}:
 {Providence, RI},
 {2024};
 {xv+927~pp.};
 doi:~{10.1090/coll/067}.


\bibitem{GrnMkh2010}
 {Goriunov, A.; Mikhailets, V.}
 {Regularization of singular Sturm-Liouville equations}.
 {\textit{Methods Funct. Anal. Topology}}
 {\textbf{2010}},
 {16},
 {no.~2},
 {120--130};
 {\MR{2667807}}.


\bibitem{GrnMkhPnk2013}
 {Goriunov, A.; Mikhailets, V.; Pankrashkin, K.}
 {Formally self-adjoint quasi-differential operators and boundary-value problems}.
 {\textit{Electron. J. Differential Equations}}
 {\textbf{2013}},
 {No.~101},
 {16~pp.};
 {\MR{3065054}}.


\bibitem{EcGsNcTs2013}
 {Eckhardt, J.; Gesztesy, F.; Nichols, R.; Teschl, G.}
 {Weyl--Titchmarsh theory for Sturm--Liouville operators with distributional potentials}.
 {\textit{Opuscula Math.}}
 {\textbf{2013}},
 {33},
 {n.~3},
 {467--563};
 doi:~{10.7494/OpMath.2013.33.3.467}.


\bibitem{EvMr1999}
 {Everitt, W.; Markus, L.}
 {\textit{Boundary value problems and symplectic algebra for ordinary differential and quasi-differential operators}};
 {Mathematical Surveys and Monographs},
 {\textbf{61}};
 {American Mathematical Society}:
 {Providence, RI},
 {1999};
 {xii+187~pp.};
 {\MR{1647856}}.


\bibitem{Hrtmn1948}
 {Hartman, P.},
 {Differential equations with non-oscillatory eigenfunctions}.
 {\textit{Duke Math. J.}}
 {\textbf{1948}},
 {15},
 {697--709};
 {\MR{0027927}}.


\bibitem{HrMk2012}
 {Hryniv, R,; Mykytyuk, Ya.}
 {Self-adjointness of {S}chr\"odinger operators with singular potentials}.
 {\textit{Methods Funct. Anal. Topology}}
 {\textbf{2012}},
 {18},
 {no.~2},
 {152--159};
 {\MR{2978191}}.


\bibitem{Ismgl1963}
 {Ismagilov, R. S.}
 {On the self-adjointness of the Sturm-Liouville operator}
 {(Russian)}.
 {\textit{Uspehi Mat. Nauk}}
 {\textbf{1963}},
 {18},
 {no.~5(113)},
 {161--166};
 {\MR{0155037}}.


\bibitem{Kato1995}
 {Kato, T.}
 {\textit{Perturbation theory for linear operators}};
 {Classics in Mathematics};
 {Springer-Verlag}:
 {Berlin},
 {1995};
 {xxii+619~pp.};
 {\MR{1335452}}.


\bibitem{KsMlNc2022}
 {Kostenko, A,; Malamud, M,; Nicolussi, N.}
 {A {G}lazman-{P}ovzner-{W}ienholtz theorem on graphs}.
 {\textit{Adv. Math.}}
 {\textbf{2022}},
 {395},
 {Paper No.~108158},
 {30~pp.};
 doi:~{10.1016/j.aim.2021.108158}.

\bibitem{MzVrbActaMathSin2019}
 {Maz'ya, V.; Verbitsky, I. E.}
 {Accretivity of the general second order linear differential operator}.
 {\textit{Acta Math. Sin. (Engl. Ser.)}}
 {\textbf{2019}},
 {35},
 {no.~6},
 {832--852};
 doi:~{10.1007/s10114-019-8127-9}.

\bibitem{MkhGrnMlb2022}
 {Mikhailets, V.; Goriunov, A.; Molyboga, V.}
 {Povzner--Wienholtz-type theorems for Sturm--Liouville operators with singular coefficients}.
 {\textit{Complex Anal. Oper. Theory}}
 {\textbf{2022}},
 {16},
 {no.~8},
 {Paper No.~113},
 {13~pp.};
 doi:~{10.1007/s11785-022-01291-y}.


\bibitem{MkhMlb2013}
 {Mikhailets, V.; Molyboga, V.}
 {Remarks on Schr\"odinger operators with singular matrix potentials}.
 {\textit{Methods Funct. Anal. Topology}}
 {\textbf{2013}},
 {19},
 {no.~2},
 {161--167};
 {\MR{3098494}}.


\bibitem{MkhMlb2018}
 {Mikhailets, V.; Molyboga, V.}
 {Schr\"odinger operators with measure-valued potentials: semiboundedness and spectrum}.
 {\textit{Methods Funct. Anal. Topology}}
 {\textbf{2018}},
 {24},
 {no.~3},
 {240--254};
 {\MR{3860804}}.




\bibitem{MkhSbl1999}
 {Mikhailets, V.; Sobolev, A.}
 {Common eigenvalue problem and periodic Schr\"odinger operators}.
 {\textit{J. Funct. Anal.}}
 {\textbf{1999}},
 {165},
 {no.~1},
 {150--172};
 doi:~{10.1006/jfan.1999.3406}.




\bibitem{NmrkRus1969}
 {Naimark, M.}
 {\textit{Linear Differential Operators}}
 {(Russian)},
 {2nd ed.};
 {With an appendix by V. \`E. Ljance};
 {Nauka}:
 {Moscow},
 {1969};
 {526~pp.};
 {\MR{0353061}}.


\bibitem{Olvr1961}
 {Olver, F. W. J.}
 {Error bounds for the {L}iouville-{G}reen (or {$WK\,B$}) approximation}
 {\textit{Proc. Cambridge Philos. Soc.}},
 {\textbf{1961}},
 {57},
 {790--810};
 doi: {10.1017/S0305004100035945}.


\bibitem{OlvrASF1997}
 {Olver, F. W. J.}
 {\textit{Asymptotics and Special Functions}};
 {AKP Classics},
 {Reprint of the 1974 original [Academic Press, New York; MR0435697 (55 \#8655)]}
 {A K Peters, Ltd., Wellesley, MA},
 {1997};
 {xviii+572~pp.};
 {\MR{1429619}}


\bibitem{Orchk1988}
 {Orochko, Yu.\ B.}
 {The hyperbolic equation method in the theory of operators of {S}chr\"odinger type with locally integrable potential}.
 {\textit{Uspekhi Mat. Nauk}}
 {\textbf{1988}},
 {43},
 {no.~2(260)},
 {43--86, 230};
 doi: {10.1070/RM1988v043n02ABEH001728}.


\bibitem{Orchk1974}
 {Oro\v cko, Ju.\ B.}
 {Sufficient conditions for the selfadjointness of a {S}turm-{L}iouville operator}
 {(Russian)}.
 {\textit{Mat. Zametki}}
 {\textbf{1974}},
 {15},
 {271--280};
 {\MR{0375002}}.


\bibitem{Phl1957}
 {Phillips, R. S.}
 {Dissipative operators and hyperbolic systems of partial differential equations}.
 {\textit{Trans. Amer. Math. Soc.}}
 {\textbf{1959}},
 {90},
 {193--254};
 doi:~{10.2307/1993202}.


\bibitem{Pvznr1953}
 {Povzner, A. Ya.}
 {On the expansion of arbitrary functions in eigenfunctions of the operator $-\Delta u+cu$}
 {(Russian)}.
 {\textit{Mat. Sbornik N.S.}}
 {\textbf{1953}},
 {32/74},
 {109--156};
 {\MR{0053330}}.


\bibitem{ReedSimon-book02_Eng_1975}
 {Reed, M.; Simon, B.}
 {\textit{Methods of modern mathematical physics. II. Fourier analysis, self-adjointness}};
 {Academic Press [Harcourt Brace Jovanovich, Publishers]}:
 {New York - London},
 {1975};
 {xv+361~pp.};
 {\MR{0493420}}.


\bibitem{Rllch1951}
 {Rellich, F.}
 {Halbbeschr\"ankte gew\"ohnliche Differentialoperatoren zweiter Ordnung}.
 {\textit{Math. Ann.}}
 {\textbf{1951}},
 {122},
 {343--368};
 doi:~{10.1007/BF01342848}.

\bibitem{Sears1950}
 {Sears, D. B.}
 {Note on the uniqueness of the Green's functions associated with certain differential equations}.
 {\textit{Canad. J. Math.}}
 {\textbf{1950}},
 {2},
 {314--325};
 doi:~{10.4153/cjm-1950-029-9}.

\bibitem{Schm2012}
 {Schm\"{u}dgen, K.}
 {\textit{Unbounded self-adjoint operators on Hilbert space}};
 {Graduate Texts in Mathematics},
 {265};
 {Springer}:
 {Dordrecht},
 {2012};
 {xx+432~pp.};
 doi:~{10.1007/978-94-007-4753-1}.


\bibitem{Shin1943}
 {Shin, D.}
 {Quasi-differential operators in Hilbert space}.
 {(Russian)},
 {Rec. Math. [Mat. Sbornik] N.S.}
 {\textbf{1943}},
 {13/55},
 {39--70};
 {\MR{0011534}}.




\bibitem{Shbn2001}
 {Shubin, M.}
 {Essential self-adjointness for semi-bounded magnetic {S}chr\"odinger operators on non-compact manifolds}.
 {\textit{J. Funct. Anal.}}
 {\textbf{2001}},
 {186},
 {no.~1},
 {92--116};
 doi:~{10.1006/jfan.2001.3778}.






\bibitem{StHnsn1966}
 {Stetkaer-Hansen, H.}
 {A generalization of a theorem of Wienholtz concerning essential selfadjointness of singular elliptic operators}.
 {\textit{Math. Scand.}}
 {\textbf{1966}},
 {19},
 {108--112};
 doi:~{10.7146/math.scand.a-10798}.


\bibitem{Wei1987}
 {Weidmann, J.}
 {\textit{Spectral theory of ordinary differential operators}};
 {Lecture Notes in Mathematics},
 {1258},
 {Springer-Verlag}:
 {Berlin},
 {1987};
 {vi+303~pp.};
 doi:~{10.1007/BFb0077960}.

\bibitem{Weyl1910}
 {Weyl, H.}
 {\"Uber gew\"ohnliche Differentialgleichungen mit Singularit\"aten und die zugeh\"origen Entwicklungen willk\"urlicher Funktionen}
 {(German)}.
 {\textit{Math. Ann.}}
 {\textbf{1910}},
 {68},
 {no.~2},
 {220--269};
 doi:~{10.1007/BF01474161}.

\bibitem{Wnhltz1958}
 {Wienholtz, E.}
 {Halbbeschr\"ankte partielle Differentialoperatoren zweiter Ordnung vom elliptischen Typus}.
 {\textit{Math. Ann.}}
 {\textbf{1958}},
 {135},
 {50--80};
 doi:~{10.1007/BF01350827}.


\bibitem{Zttl1975}
 {Zettl, A.}
 {Formally self-adjoint quasi-differential operator}.
 {Rocky Mountain J. Math.}
 {\textbf{1975}},
 {5},
 {no.~3},
 {453--474};
 doi:~{10.1216/RMJ-1975-5-3-453}.


\bibitem{Zttl2005}
 {Zettl, A.}
 {\textit{Sturm--Liouville Theory}};
 {Mathematical Surveys and Monographs},
 {121};
 {American Mathematical Society}:
 {Providence, RI},
 {2005};
 {xii+328~pp.};
 doi:~{10.1090/surv/121}.


\bibitem{Zttl2021}
 {Zettl, A.}
 {\textit{Recent developments in {S}turm-{L}iouville theory}};
 {De Gruyter Studies in Mathematics},
 {76},
 {De Gruyter}:
 {Berlin},
 {2021};
 {xii+245~pp.};
 doi:~{10.1515/9783110719000}.


\end{thebibliography}
\end{document}